\documentclass[11pt]{article}
\usepackage{enumerate}
\usepackage{amsfonts,amsmath,amssymb,amsthm,bbm,bm}
\usepackage{mathrsfs}
\usepackage{mathtools}
\usepackage[dvipsnames]{xcolor}
\usepackage{soul}
\setstcolor{red}
\usepackage{graphicx}
\usepackage{caption}
\usepackage{subcaption}
\usepackage{float}
\usepackage{tikz}
\usepackage{tikz-cd}
\usetikzlibrary{matrix}
\numberwithin{equation}{section} 

\usepackage[style=alphabetic,backend=biber, sorting = nyt, maxalphanames=99,minalphanames=99,maxbibnames=99]{biblatex}
\usepackage{sepfootnotes}
\usepackage[hyperfootnotes=false]{hyperref}

\usepackage{thmtools}
\usepackage{hyperref}
\usepackage[capitalise]{cleveref}

\newtheorem{theorem}{Theorem}[section]
\newtheorem{corollary}[theorem]{Corollary}

\newtheorem{lemma}[theorem]{Lemma}
\newtheorem{proposition}[theorem]{Proposition}

\theoremstyle{remark}
\newtheorem{remark}[theorem]{Remark}

\theoremstyle{definition}
\newtheorem{definition}[theorem]{Definition}
\newtheorem{example}[theorem]{Example}

\newcommand{\E}{\mathbb{E}}
\newcommand{\R}{\mathbb{R}}

\newcommand{\bL}{\mathbf{L}}

\newcommand{\bO}{\mathbf{O}}

\newcommand{\calF}{\mathcal{F}}

\newcommand{\W}{\mathcal{W}}
\newcommand{\AW}{\mathcal{AW}}

\newcommand{\cpl}{\mathrm{Cpl}}

\newcommand{\tr}{\mathrm{tr}}

\newcommand{\diag}{\mathrm{diag}}

\newcommand{\fp}[1]{{\mathbb #1}}

\newcommand{\cbullet}{\vcenter{\hbox{\scalebox{0.4}{$\bullet$}}}}

\title{The geometry of the adapted Bures--Wasserstein space}
\author{Beatrice Acciaio\thanks{Department of Mathematics, ETH Z\"{u}rich, Switzerland. \emph{beatrice.acciaio@math.ethz.ch}}, Daniel Bartl\thanks{Department of Mathematics, Department of Statistics and Data Science, National University of Singapore, Singapore. \emph{bartld@nus.edu.sg}}, Anne Grass\thanks{Department of Mathematics, National University of Singapore, Singapore. \emph{grassa@nus.edu.sg}}, Songyan Hou\thanks{Department of Mathematics, ETH Z\"{u}rich, Switzerland. \emph{songyan.hou@math.ethz.ch}}, Gudmund Pammer\thanks{Department of Mathematics, TU Graz, Austria. \emph{gudmund.pammer@tugraz.at}}}

\date{\today}
\begin{document}
\maketitle

\begin{abstract}
    The adapted Bures–Wasserstein space consists of Gaussian processes endowed with the adapted Wasserstein distance. It can be viewed as the analogue of the classical Bures--Wasserstein space in optimal transport for the setting of stochastic processes, where the standard Wasserstein distance is inadequate and has to be replaced by its adapted counterpart.

    We develop a comprehensive geometric theory for the adapted Bures--Wasserstein space, thereby also providing the first results on the fine geometric structure of adapted optimal transport. In particular, we show that the adapted Bures--Wasserstein space is an Alexandrov space with non-negative curvature and provide explicit descriptions of tangent cones and exponential maps. 
    Moreover, we show that Gaussian processes satisfying a natural non-degeneracy condition form a geodesically convex subspace. 
    This subspace is characterized precisely by the property that its tangent cones are linear and hence coincide with the tangent space.
\end{abstract}

\section{Introduction and main results}
The theory of optimal transport provides a powerful and unifying framework for comparing probability measures through distances that reflect the geometry of the ambient space. A cornerstone of this theory is the 2-Wasserstein distance \(\mathcal W_2\), which equips the space \(\mathcal P_2(\mathbb R^d)\) of square integrable probability measures with a natural and well-behaved metric. For \(\mu,\nu\in\mathcal P_2(\mathbb R^d)\), it is defined as
\[
\mathcal W_2^2(\mu,\nu):=\inf_{X\sim\mu,\,Y\sim\nu}\mathbb E[|X-Y|^2].
\]
A fundamental insight is that the resulting metric space $(\mathcal P_2(\R^d),\mathcal W_2)$ allows for a geometric interpretation: it can be viewed both as a complete geodesic space with non-negative Alexandrov curvature \cite{sturm2006geometry,lott2009ricci} and, via the Otto calculus \cite{otto2001geometry}, formally as an infinite-dimensional Riemannian manifold. 
This viewpoint has driven major advances across partial differential equations, geometric analysis, probability theory, and data science; see, e.g., \cite{
AmGiSa08,
ChNiRi25,
figalli2021invitation,
peyre2019computational,
rahimian2019distributionally,
villani2008optimal}.

While the full Wasserstein geometry is often numerically and analytically challenging, restricting $\mathcal W_2$ to Gaussian measures yields the Bures--Wasserstein distance, which turns out to be highly tractable and central to   many applications; see, e.g., \cite{bunne2023schrodinger,chewi2020gradient,ChNiRi25,diao2023forward}.
When $\mu$, $\nu$ are centered Gaussian measures with covariance matrices $\Sigma_0, \Sigma_1 \in {\bf S}^d_+$, that is, in the cone of symmetric and positive semi-definite matrices, the transport problem admits an explicit, closed-form solution and endows ${\bf S}_+^d$ with the natural metric $d_{\rm BW}$,  \cite{dowson1982frechet,olkin1982distance}:
\[
  \W_2^2(\mu,\nu) =    {\rm tr}(\Sigma_0) + {\rm tr}(\Sigma_1) - 2{\rm tr}\Big( \big( \Sigma_0^{1/2} \Sigma_1 \Sigma_0^{1/2}\big)^{1/2}\Big)
  =:d_{\rm BW}^2(\Sigma_0,\Sigma_1).
\]
Historically, this distance has been studied independently in quantum theory and statistics.
Gelbrich \cite{gelbrich1990formula} famously identified that the Bures metric on covariance matrices coincides with the Wasserstein distance between Gaussian measures.
The geometric structure of the Bures--Wasserstein space $({\bf S}_+^d,d_{\rm BW})$ was first rigorously investigated by Takatsu \cite{Ta11}, 
who gave a complete characterization of the tangent cones.

However, when one is interested in stochastic processes or time-series data, both the Wasserstein and the Bures--Wasserstein distance fail to capture important features.
Indeed, if $X = (X_t)_{t = 1}^T$ and $Y = (Y_t)_{t = 1}^T $ are discrete-time stochastic processes with paths in $(\R^d)^T$, 
then the usual Wasserstein distance fails to capture how information is progressively revealed over time by the respective canonical filtrations.
see \cite[Figure~1]{backhoff2020adapted} for an illustrative example.
This leads to discontinuity of virtually any stochastic control problem with respect to $\W_2$, such as pricing and hedging, optimal stopping and utility maximization.

\begin{remark}
An equivalent formulation of $\W_2$, better suited to the setting of stochastic processes, is as follows.
Let $(\Omega,\mathcal{F},\mathbb{P})$ be a standard  probability space and define (with slight abuse of notation) the pseudo-distance on $L^2(\R^d)$ by
\begin{equation}
    \label{intro:W}
\mathcal{W}_2^2(X,Y) := \inf_{X'\sim X,\, Y'\sim Y} \E\bigl[|X'-Y'|^2\bigr].
\end{equation}
Then, clearly, $(\mathcal{P}_2(\R^d),\W_2)$ is isometric to the quotient space
$(L^2(\R^d)/_\sim,\W_2)$.
\end{remark}

\subsection{Background on adapted optimal transport}

The theory of adapted optimal transport proposes a structural resolution to address the weaknesses of classical optimal transport in the setting of stochastic processes by adequately taking into account the information encoded in the filtration.
To streamline the exposition, we introduce the adapted Wasserstein distance $\AW_2$ analogous to the Wasserstein distance $\W_2$ in \eqref{intro:W}, in a probabilistic manner.
We consider $d$-dimensional processes in discrete time $t=1,\dots,T$ and identify them with $dT$-dimensional random variables. 
With this in mind, let $(U_t)_{t = 1}^T$ be a sequence of independent, continuous random variables on a standard probability space $(\Omega,\mathcal F,\mathbb P)$,
let $(\mathcal F_t)_{t = 1}^T$ denote their generated canonical filtration  $\mathcal{F}_t=\sigma(U_s:s\leq t )$, and define $L^2_{\rm ad}((\R^d)^T)$ as the set of square-integrable, $(\mathcal F_t)_{t = 1}^T$-adapted, $\R^{dT}$-valued stochastic processes.

Two processes $Z=(Z_t)_{t=1}^T, Z'=(Z'_t)_{t=1}^T \in L^2_{\rm ad}((\R^d)^T)$ are said to have the same adapted distribution (denoted by $Z\sim_{\rm ad}Z'$) if their nested conditional laws coincide:
\begin{equation*}
    \label{eq:nested.condi}
    {\rm law}\big( {\rm law} (\dots {\rm law}(Z | \mathcal F_{T-1}) \dots  |\mathcal F_1 ) \big) =
    {\rm law}\big( {\rm law} (\dots {\rm law}(Z' | \mathcal F_{T-1}) \dots  |\mathcal F_1 ) \big).
\end{equation*}
By formalizing which assertions belong to the language of probability, Hoover and Keisler \cite{hoover1984adapted} made precise what it means for two processes  $Z$ and $Z'$ to have the same probabilistic properties, and showed that this holds if and only if $Z\sim_{\rm ad} Z'$.
With this in mind, the adapted Wasserstein distance is then defined as
\begin{equation*}
    \mathcal{AW}_2^2(X,Y) := \inf_{X' \sim_{\rm ad} X, \, Y' \sim_{\rm ad} Y} \mathbb E\big[ |X' - Y'|^2 \big].
\end{equation*}
The space $(L^2_{\rm ad}((\R^d)^T)/_{\sim_{\rm ad}}, \AW_2)$ is called the Wasserstein space of stochastic processes.
In particular, for $T=1$ this space is isometric to the Wasserstein space $(L^2(\R^d)/_{\sim},\W_2)$, 
but for $T>1$ the topology induced by $\mathcal{AW}_2$ is strictly finer than the one induced by $\W_2$: 
it metrizes the so-called adapted weak topology, which happens to be the correct topology for dynamic stochastic optimization problems, in the sense that it is the initial topology for basic dynamic optimization problems; see \cite{BaBePa21,backhoff2020all}.
We also refer to \cite{acciaio2024time,acciaio2020causal,bartl2021sensitivity,blanchet2024empirical,bonnier2023adapted,cont2024causal,hitz2024bicausal,jiang2024sensitivity,lassalle2018causal,pflug2014multistage,sauldubois2024first}.

Developing a geometric understanding for adapted optimal transport is essential to bringing this theory to the level of its classical counterpart.
However this direction remained largely open; at present, the only related result is an abstract Brenier theorem \cite{beiglbock2025brenier}, which characterizes optimal transports via convex potentials of law-invariant functionals on high-dimensional nested \(L^2\) spaces and establishes the existence of suitable reference measures (taking the role of the Lebesgue measure in classical optimal transport). However, many important classes of processes -- including the Gaussian processes studied here -- are singular with respect to these reference measures. 
This paper addresses this gap by developing a comprehensive geometric framework in the tractable setting of Gaussian processes.

\subsection{The adapted Bures--Wasserstein space}

Fix a stochastic basis $(\Omega,\mathcal{F}, \mathbb{P})$ supporting $T$ independent $d$-dimensional standard Gaussians 
$G_{t}\sim\mathcal{N}(0,{\rm Id})$, $t=1,\dots, T$ (here ${\rm Id}$ is the identity matrix on $\R^{d\times d}$), and set 
$\mathcal{G}_{t}:=  \sigma(G_{s}:s\leq t)$ for $t=1,\dots,T$.
In the present context of stochastic processes, it is natural to parametrize covariance matrices through block-lower triangular matrices
\[ \mathbf{L} := 
\left\{ \begin{pmatrix}
L_{1,1} & 0  &  0& \cdots & 0  \\
L_{2,1} & L_{2,2} & 0& \cdots & 0 \\
\vdots   & \vdots   & \cdots &\ddots & \vdots   \\
L_{T,1} & L_{T,2} &\cdots &  \cdots & L_{T,T}
\end{pmatrix} : 
L_{t,s} \in \mathbb{R}^{d \times d}\right\}.\]
Clearly, every $L\in{\bf L}$ gives rise to a process $X=(X_{t})_{t=1}^{T}$ with Gaussian law, by  setting
\[ X_{t}:= (LG)_t = \sum_{s=1}^{t}L_{t,s}G_{s}\]
for every $t=1,\dots, T$.
 Conversely, for every positive semi-definite $\Sigma\in \R^{dT\times dT}$ there is $L\in {\bf L}$ such that $X:=LG$ has Gaussian law with covariance matrix $\Sigma$.

From now on, we shall identify the stochastic process $X=LG$ with  $L\in   {\bf L}$ and set
\[  d_{\rm ABW}(L,M):= \mathcal{AW}_{2}(X,Y), \]
where $X=LG$ and $Y=MG$ for $L,M \in \mathbf L$.

\begin{remark}
With minor abuse of notation, throughout this paper we will often consider processes $(X_t)_{t=1}^T$ as vectors $(X_1,\dots, X_T)$, and write e.g.\ $LG$  as above. 

Moreover, all results trivially extend to Gaussian processes with arbitrary mean
$X_t = a_t + (LG)_t$, $a_t \in \R^d$, since the $\AW_2$-distance factorizes and reduces
to the Euclidean distance on the mean component; see \Cref{lem:AW-Gaussian-distance}.
\end{remark}

It follows from \cite{acciaio2025entropic,GuWo25} (see also \Cref{lem:AW-Gaussian-distance,lem:ABW-distance} below) that \( d_{\rm ABW}(L,M) \) admits an equivalent matrix optimization formulation with a closed-form solution, namely
\begin{align}
\label{eq:formula.d}
d_{\rm ABW}^2(L,M)
= \min_{O\in \bO} \|L - MO\|_{\rm F}^2
= \|L\|_{\rm F}^2 + \|M\|_{\rm F}^2 - 2\,\mathrm{tr}(S),
\end{align}
where \( \bO \subset \mathbf L\) 
denotes the set of block-diagonal matrices with $d\times d$-dimensional orthogonal diagonal blocks,
\( \|\cdot\|_{\rm F} \) denotes the Frobenius norm, and $S \coloneqq \mathrm{diag}(S_1,\dots,S_T)$, where $S_t$ is the singular value matrix of \( (M^\top L)_{t,t} \), for all $t =1,\dots,T$.
The set of optimizers of \eqref{eq:formula.d} -- denoted by ${\bf O}^\ast(L,M)$ -- turns out to be fundamental for understanding the adapted Bures-Wasserstein geometry.
In particular, as we will explain, $\bO^*(L,M)$ is related to the optimizers of $\AW_2(LG,MG)$. 
Moreover, we show that displacement interpolations are parametrized by this set; see Section~\ref{sec:AOT} for details and \Cref{lem:ABW-optimizers} for several explicit characterizations of ${\bf O}^\ast(L,M)$.

For the remainder of the introduction, 
we focus on describing the geometric structure of $d_{\rm ABW}$, noting that all results admit straightforward analogues in the $\mathcal{AW}_2$ framework.

\begin{definition}
The adapted Bures--Wasserstein space is the quotient space 
\[({\bf L}/_{{\bf O}} ,d_{\rm ABW}) ,\]
where $L\sim_{\bf O} M $ if and only if $L=MO$ for some $O\in {\bf O}$ (which is equivalent to $d_{\rm ABW}(L,M)=0$ by \eqref{eq:formula.d}).
\end{definition}

For  $L\in {\bf L}$, we write $[L]:= \{LO : O \in \mathbf O\}$ for the corresponding equivalent class, which is an element in ${\bf L}/_{\bf O}$.
With minor abuse of notation, we use $d_{\rm ABW}$ for both the distance on $\bL$ and the quotient distance on $\bL/_{\bO}$.

The following is the first main result of this article.

\begin{theorem}
\label{thm:ABW.alex.intro}
$ ({\bf L}/_{\bf O} ,d_{\rm ABW})  $ is an Alexandrov space with non-negative curvature.
\end{theorem}

Here, a metric space is called Alexandrov space with non-negative curvature if it is locally complete, geodesic, and geodesic triangles in the space are no thinner than their Euclidean comparison triangles;
we refer to \cite{AmGiSa08,AlKaPe24,BuGrPe92}  and Sections  \ref{sec:AOT.gen} and \ref{sec:alex.background}   below for some more background. 
We only note here that Alexandrov spaces with curvature bounded from below carry a rich Riemannian-like structure -- a structure that we will explore in the following section.

\begin{remark}
\label{rem:filtration}
In contrast to prior works on the adapted Wasserstein distance between Gaussian processes (see, \cite{acciaio2025entropic,GuWo25}), we endow the processes $X$ with the Gaussian filtration  $(\mathcal{G}_t)_{t=1}^T$, rather than the canonical filtration $\mathcal{F}^{X}_t=\sigma(X_s:s\leq t)$ generated by the process itself.
Clearly, if the covariance of $X$ has full rank (equivalently, if $L$ has full rank), then  $\mathcal{F}^{X}=\mathcal{G}$.
In general, that will not be the case. 
For instance
\begin{align*}
    \label{eq:intro.example.matric}
L:=
\begin{pmatrix}
0 & 0 \\
1 & 1
\end{pmatrix}
\quad\text{and}\quad
M:=
\begin{pmatrix}
0 & 0 \\
0 & \sqrt{2}
\end{pmatrix}
\end{align*}
yield processes $X=LG$ and $Y=MG$ with the same distribution, whilst $\mathcal{L}(X_{2}\mid \mathcal{G}_{1})
= \mathcal{N}(G_{1},1)$ and $
\mathcal{L}(Y_{2}\mid \mathcal{G}_{1})
= \mathcal{N}(0,2)$.

Crucially, the present choice of the filtration turns out to be essential, as the space of Gaussian processes  endowed with the natural filtrations is not geodesic: there are situations when there is a unique geodesic between full rank $L$ and $M$ which, at intermediate times, degenerates; see \Cref{rem:full.rank.not.geo} below and \cite{acciaio2025entropic} for more details.
\end{remark}

\subsection{Adapted Bures--Wasserstein geometry for regular processes}

In the classical Bures--Wasserstein theory,  restricting to strictly positive-definite covariance matrices reveals the most elegant geometric behaviour.
The analogue is true in the present adapted setting, when working with appropriate “regular’’ subspaces. 
Here, however, the appropriate notion of “regularity’’ for subsets of ${\bf L}$ is inherently different -- shaped by the filtration rather than just the law, as we shall explain in what follows.

Define
\[ {\bf L }^{\rm reg} := \left\{ L\in {\bf L} : (L^{\top } L )_{t,t} \text{ is positive definite for all } t=1,\dots, T\right\}.\]
If $L\in {\bf L}$ has full rank, then $L\in {\bf L}^{\rm reg}$, but ${\bf L }^{\rm reg}$ contains many more elements:

\begin{example}
In the setting of Remark \ref{rem:filtration} and using its notation, we have that $L \in {\bf L }^{\rm reg}$ while $M\notin  {\bf L }^{\rm reg}$.
Moreover, recall that $\mathcal{L}(X_{2}|\mathcal{G}_{1}) = \mathcal{N}(G_{1},1)$ and $\mathcal{L}(Y_{2}|\mathcal{G}_{1}) = \mathcal{N}(0,2)$, hence the filtration $(\mathcal{G}_t)_{t=1}^T$ contains some information about the process $X$, while for $Y$ the $G_{1}$-coordinate is never used. 
\end{example}

Whereas the subset of full-rank matrices in ${\bf L}$ fails to form a geodesic space (see Remark \ref{rem:full.rank.not.geo}), not only does ${\bf L}^{\rm reg}/_{\bf O}$ form a geodesically complete subset,  its infinitesimal geometry can be described explicitly and tractably.
Indeed, recall that since ${\bf L}/_{\bf O}$ is an Alexandrov space with non-negative curvature, 
tangent cones exist at every point (see Section~\ref{sec:alex.background}); 
we denote them by 
\[
({\bf T}([L]), d_{{\bf T}([L])}),
\]
where $d_{{\bf T}([L])}$ denotes the canonical metric on the tangent cone, 
for $[L]\in {\bf L}/_{\bf O}$.
To formulate their structure, for $L\in{\bf L}$ set 
\[ 
\mathcal{V}(L) = \left\{ V \in {\bf L}: (V^\top L)_{t,t} \text{ is symmetric for all } t = 1, \dots, T \right\}.
\]
The following is our second main result:

\begin{theorem}
\label{thm:L.reg.convex.subset.L.intro}
	${\bf L }^{\rm reg}/_{\bf O}$ is a geodesically convex subset of ${\bf L}/_{\bf O}$.
    Moreover, for every $L\in{\bf L}^{\rm reg} $, there exists an isometry 
    \[ \iota_L \colon   \left( {\bf T}([L]) , d_{{\bf T}([L])} \right) \to (\mathcal{V}(L) , \|\cdot\|_{\rm F}).\]
\end{theorem}

We stress that Theorem~\ref{thm:L.reg.convex.subset.L.intro} constitutes the first result of this kind for the adapted Wasserstein distance.
The mapping $\iota_L$ fixes the representative $L$ of the class $[L]$ and 
identifies all elements relative to this reference.
In particular, if $O\in{\bf O}$ and $M = LO$ 
so that $M\in[L]$, then $\iota_M(\cdot) = \iota_L(\cdot ) O$.

Moreover, it turns out that the exponential and logarithmic maps admit a particularly simple description in this setting.
To that end, recall that
\[
\exp_{[L]} \colon {\bf T}([L]) \longrightarrow {\bf L}/_{\bf O} 
\]
denotes the exponential map at the equivalence class $[L] \in {\bf L}/_{\bf O}$, sending a tangent vector to the endpoint of the (unique) constant-speed $d_{\rm ABW}$-geodesic starting at $[L]$ with the given initial velocity.
Note that $\exp_{[L]}$ may not be defined on the entire space ${\bf T}([L])$ because constant speed geodesics may not be indefinitely extensible.

\begin{proposition}
\label{prop:L.reg.exponential.intro}
For every $L\in{\bf L}^{\rm reg} $ and $V\in\mathcal{V}(L)$ there is $r_0>0$ such that, for all $r\in[0,r_0]$,
	\[\exp_{[L]} (\iota_L^{-1}(rV)) = [L+rV].\]
\end{proposition}

Again, $\iota_L$ simply fixes the representative $L$ of $[L]$.
The following diagram should help to explain the structure:
\begin{center}
\begin{tikzcd}
    ({\bf T}([L]), d_{{\bf T}([L])}) \arrow[r, "\exp_{[L]}"] \arrow[d, "\iota_L"'] 
    & ({\bf L}/_{\bf O},d_{\rm ABW})  \\
        (\mathcal V(L),\| \cdot \|_{\rm F}) \arrow[r, "\exp_L"] 
    & (\mathbf L, \| \cdot \|_{\rm F}) \arrow[u, "{[\,\cdot\,]}"']
\end{tikzcd}
\end{center}
Here, $\exp_L$ is the  exponential map on the flat space $(\mathbf L,\|\cdot\|_{\rm F})$ which is given by $\exp_L(V) = L+V$.
The diagram shows that the exponential map $\exp_{[L]}$ on the factor space $\mathbf{L}/_{\mathbf O}$ is consistent with the isometry $\iota_L$ on the tangent space and the exponential map $\exp_{L}$ on the underlying flat space, that is, $[\exp_L \circ \iota_L] = \exp_{[L]}$.

We next describe the logarithmic map, the inverse of the exponential map:
\[
\log_{[L]} \colon {\bf L}/_{\bf O} \longrightarrow {\bf T}([L]).
\]
Whenever a unique geodesic from $[L]$ to $[M]$ exists, $\log_{[L]}([M])$ is given by the initial velocity of that geodesic.
As in the case of the exponential map, $\log_{[L]}([M])$ need not be defined for all $[M]\in \mathbf{L}/_\mathbf{O}$. The following result therefore characterizes when the logarithmic map exists and provides its explicit form.
To this end, recall that $\mathbf{O}^\ast(L,M)\subset \mathbf{O}$ denotes the set of optimizers of~\eqref{eq:formula.d}.

\begin{theorem} \label{thm:L.reg.uniq.geodesic}
\label{thm:L.reg.log.intro}
	For every $L\in {\bf L}^{\rm reg}$ and $M\in {\bf L} $, the following are equivalent:
    \begin{enumerate}[(i)]
        \item  There exists a unique geodesic between $[L]$ and $[M]$;
        \item  ${\rm rk}((M^\top M)_{t,t}) = {\rm rk}((M^\top L)_{t,t})$ for all $t=1,\dots,T$.
    \end{enumerate}
    If either is satisfied,  $\log_{[L]}([M])$ exists and is given by 
    \[ \log_{[L]}([M]) = \iota_L^{-1}( MP -L),\]
 	for some $P\in {\bf O}^\ast(L,M)$.
\end{theorem}

In particular, in the setting of \Cref{thm:L.reg.uniq.geodesic}, the geodesic from $[L]$ to $[M]$ is given by $([L+u(MP-L)])_{u\in[0,1]}$, 
for $P\in {\bf O}^\ast(L,M)$.

We end this section on regular $L$'s by giving a simple and practically relevant sufficient condition for the existence of the log map.
To state it, for a symmetric matrix $A$, we set $\lambda_{\min}(A)$ to be its smallest eigenvalue.

\begin{corollary} \label{cor:L.reg.uniq.geodesic}
    For every $L\in {\bf  L}^{\rm reg}$ and $M\in {\bf L} $ satisfying 
    \[ d_{\rm ABW}^{2}([L],[M]) < \min_{t=1,\dots, T}  \lambda_{\min}\left((L^{\top}L)_{t,t}\right),\] there exists a unique geodesic between $[L]$ and $[M]$.
\end{corollary}

\begin{remark}
We briefly comment on the apparent discrepancy between the structure of
\(\mathcal{V}(L)\) and the tangent space in the classical Bures--Wasserstein setting.
Consider \(T=1\), so that \(d_{\mathrm{ABW}} = d_{\mathrm{BW}}\) and the adapted and
classical Bures--Wasserstein spaces coincide. Let \(L\) be full rank and set
\(\Sigma = L^\top L\). In the present formulation, the tangent space is
$
\mathcal{V}(L)
\;=\;
\{\, V \in \mathbb{R}^{d\times d} : V^\top L \text{ is symmetric} \,\},
$
whereas in the classical formulation the tangent space at \(\Sigma\) is
$
\mathcal{W}(\Sigma)
\;=\;
\{\, W \in \mathbb{R}^{d\times d} : W \text{ is symmetric} \,\}.
$

This difference is solely due to a change of parametrization. 
Indeed, classical Bures--Wasserstein geodesics are given by
\[
[0,1]\ni u\mapsto
\Sigma^{1/2} + u (W-\mathrm{Id}) \Sigma^{1/2}
\]
corresponding to the law of $X+ u(W-{\rm Id})X$
where $X\sim\mathcal{N}(0,\Sigma)$,
while in the present setting geodesics take the form
\[
[0,1]\ni u\mapsto L + u (V-L)
\]
corresponding to the law of  $LG+ u(V-L)G = X + u(VL^{-1}-{\rm Id})X$.
In other words, in the $d_{\mathrm{BW}}$ setting, geodesics are given by displacement interpolations driven by optimal transport maps acting on $X$. 
In contrast, in the present setting the geodesics are defined via optimal transport maps acting on $G=L^{-1}X$. 
Defining $W := V L^{-1}$ establishes a one-to-one correspondence between the two parametrizations, as $V L^{-1}$ is symmetric if and only if $V^\top L$ is symmetric.
\end{remark}

\subsection{Adapted Bures--Wasserstein geometry for irregular processes}

For general (non-regular) $L\in \mathbf{L}$, a full geometric theory is still available, but the characterization of tangent structures becomes more subtle. 
In this setting, one no longer obtains genuine tangent spaces; instead, one must work  with tangent cones (see Section~\ref{sec:alex.background}).

To formulate our results, for $V,W \in \bL$, set 
\[ d_{{\bf O}^{\ast}(L,L)}(V,W) := \inf_{ O \in {\bf O}^{\ast}(L,L) } \|V- WO\|_{\rm F}.\]
We endow $\mathcal{V}(L)$ with the pseudo-distance $d_{{\bf O}^{\ast}(L,L)}$, denote the corresponding quotient space  by $\mathcal{V}(L)/_{{\bf O}^{\ast}(L,L)}$, and write $[V]_{{\bf O}^{\ast}(L,L)}$ for elements in that space.

\begin{theorem}
\label{thm:tangent.intro}
        For every $L\in{\bf L}$,  there is an isometry
        \[ \iota_L\colon  \left( {\bf T}([L]) , d_{{\bf T}([L])} \right) 
        \to \left(\mathcal{V}(L)/_{{\bf O}^{\ast}(L,L)}, d_{{\bf O}^{\ast}(L,L)}\right).\]
\end{theorem}

We shall see that  ${\bf O}^\ast(L,L)=\{\rm Id\}$ for $L\in{\bf L}^{\rm reg}$ and hence $d_{{\bf O}^\ast(L,L)}$ is simply the metric induced by $\|\cdot\|_{\rm F}$ in that case; thus \Cref{thm:tangent.intro} is consistent with \Cref{thm:L.reg.convex.subset.L.intro}.
It turns out that the descriptions of the exponential and logarithmic map are similar to those in the regular case, modulo appropriate equivalence classes.

\begin{theorem}
\label{thm:exponential.general.intro}
For every $L\in{\bf L}$ and $V\in \mathcal{V}(L)$, 
 \[ \exp_{[L]}( \iota_L^{-1}([V]_{{\bf O}^{\ast}(L,L)})) = [L+V] .\]
	Moreover, for every $L,M\in {\bf L}$, the following are equivalent:
    \begin{enumerate}[(i)]
        \item  There exists a unique geodesic between $[L]$ and $[M]$;
        \item  ${\rm rk}((M^\top M)_{t,t}) = {\rm rk}((M^\top L)_{t,t})$ for all $t=1,\dots,T$.
    \end{enumerate}
    If either is satisfied,  $\log_{[L]}([M])$ exists and is given by 
    \[ \log_{[L]}([M]) =  \iota_L^{-1}\left( [MP -L]_{{\bf O}^{\ast}(L,L)} \right),\]
 	for some $P\in {\bf O}^\ast(L,M)$.
\end{theorem}

As already indicated, all the results mentioned in this introduction for $({\bf L}/_{\rm O},d_{\rm ABW})$ naturally have analogues in the language of adapted optimal transport.
We provide the background needed on adapted optimal transport and key additional results in the next section.

\section{Adapted optimal transport} \label{sec:AOT}

In the current paper, we opted to introduce the adapted Wasserstein distance for processes on a fixed filtered probability space.
We find that this formulation is conceptually simpler than the classical one and, in particular, more accessible to a  broader audience.
Indeed, recall that
\begin{equation*}
    \label{eq:def.AW.recall}
    \mathcal{AW}_2^2(X,Y) := \inf_{X' \sim_{\rm ad} X, \, Y' \sim_{\rm ad} Y} \mathbb E\left[ |X' - Y'|^2 \right]
\end{equation*}
where $X',Y'$ are defined on a fixed standard probability space $(\Omega, \calF, \mathbb P)$ with filtration $(\mathcal{F}_t)_{t=1}^T$ generated by independent continuous random variables $(U_t)_{t=1}^T$.
We now put this definition in relation with the classical definition of $\mathcal{AW}$ in terms of bicausal couplings.

The adapted Wasserstein distance, as introduced in \cite{BaBePa21}, is defined for 5-tuples
\[
    \fp X = \left( \Omega^\fp X, \calF^\fp X, \mathbb P^\fp X, (\calF_t^\fp X)_{t = 1}^T, X \right),
\]
comprising a filtered probability space $(\Omega^\fp X, \calF^\fp X, \mathbb P^\fp X, (\calF_t^\fp X)_{t = 1}^T)$ and an $(\calF_t^\fp X)_{t = 1}^T$-adapted stochastic process $X = (X_t)_{t = 1}^T$.
We denote the class of all such filtered processes by $\mathcal{FP}$, and by $\mathcal{FP}_2$ the subclass of processes with a finite second moment.
A coupling $\pi$ between two filtered processes $\fp X$ and $\fp Y$ is a probability measure on the product space $(\Omega^\fp X \times\Omega^\fp Y, \calF^\fp X \otimes \calF^\fp Y)$ having marginals $\mathbb P^\fp X$ and $\mathbb P^\fp Y$.
Such a coupling $\pi$ is termed bicausal -- denoted by $\pi \in \cpl_{\rm bc}(\fp X,\fp Y)$ -- if it satisfies the conditional independence properties:
\begin{align*}
    \calF^\fp X_T \text{ is independent of }\calF^\fp Y_t \text{ given }\calF_t^\fp X\text{ under }\pi, \\
    \calF^\fp Y_T \text{ is independent of }\calF^\fp X_t \text{ given }\calF_t^\fp Y\text{ under }\pi,
\end{align*}
for all $t \in \{1, \dots, T-1\}$. 
Note that we tacitly identify sub-$\sigma$-algebras $\cal A \subseteq \calF^{\fp X}$ (resp.\ $\cal B \subseteq \calF^{\fp Y}$) with their cylindrical extensions $\mathcal A \otimes \{\emptyset,\Omega^{\fp Y}\}$ (resp.\ $\{\emptyset,\Omega^{\fp X}\} \otimes \mathcal B$) on the product space.
The adapted Wasserstein distance $\mathbb{AW}_2$ between $\fp X$ and $\fp Y$ is defined in \cite{BaBePa21} by
\begin{equation}
    \label{eq:classical.AW}
    \mathbb{AW}_2^2(\fp X,\fp Y) := \inf_{\pi \in \cpl_{\rm bc}(\fp X,\fp Y)} \E_\pi\left[\|X-Y\|^2\right].
\end{equation}
In that paper it is established that $\mathbb{AW}_2$ constitutes a semi-metric on $\mathcal{FP}_2$ and that the Wasserstein space of stochastic processes $({\rm FP}_2,\AW_2)$, defined as the quotient space of $\mathcal{FP}_2$ with respect to the equivalence relation $\sim_{\mathbb{AW}}$ endowed with the quotient metric, 
is a Polish metric space. Importantly, the infimum in \eqref{eq:classical.AW} is attained. 

Rather than working with abstract filtered processes, we adopt a ``Lagrangian'' perspective by fixing a standard probability space $(\Omega, \calF, \mathbb P)$ equipped with a filtration $(\calF_t)_{t=1}^T$.
Let $L_{\rm ad}^2((\R^d)^T)$ denote the space of square-integrable, $(\calF_t)_{t=1}^T$-adapted processes on this fixed space.
Any process $X \in L_{\rm ad}^2((\R^d)^T)$ naturally induces an element in $\mathcal{FP}_2$ via the embedding:
\[
    \iota_{\mathcal{FP}}(X) := \big( \Omega, \calF, \mathbb P, (\calF_t)_{t=1}^T, X \big).
\]
As shown in \cite{beiglbock2025brenier}, the standard probability space $(\Omega, \calF, \mathbb P)$
is sufficiently rich to represent the entire Wasserstein space of stochastic processes. Specifically, for every $\fp X \in \mathcal{FP}_2$, there exists a representative $X' \in L_{\rm ad}^2((\R^d)^T)$ such that
\[
    \mathbb{AW}_2\big( \iota_{\mathcal{FP}}(X'), \fp X \big) = 0
\]
and the map $\iota_{\mathcal{FP}}$ is an isometry between $L_{\rm ad}^2((\R^d)^T)$ and the space 
${\rm FP}_2$, i.e.,
\[
    \mathbb{AW}_2(\iota_{\mathcal{FP}}(X),\iota_{\mathcal{FP}}(Y)) = \AW_2(X,Y).
\]

This correspondence extends to couplings.
Indeed, two adapted processes
$X', Y' \in L_{\rm ad}^2((\R^d)^T)$ give rise to a bicausal coupling between their induced filtered processes,
\[
    \iota_{\mathcal{FP}}(X,Y):=\big( \Omega \times \Omega, \calF \times \calF, ((\omega \mapsto (\omega,\omega))_\# \mathbb P, (\calF_t \otimes \calF_t)_{t = 1}^T, (X'_t,Y'_t)_{t = 1}^T \big).
\]
Conversely, any bicausal coupling $\pi \in \cpl_{\rm bc}(\fp X, \fp Y)$ defines a filtered process $\fp Z$ on the product space,
\[
    \fp Z := \big( \Omega^\fp X \times \Omega^\fp Y, \calF^\fp X \otimes \calF^\fp Y, \pi, (\calF_t^\fp X\otimes \calF_t^\fp Y)_{t = 1}^T, (X, Y) \big),
\]
which, in turn, admits a pair $(X',Y')$ in $L_{\rm ad}^2((\R^d \times \mathbb R^d)^T)$ with $ \mathbb{AW}_2(\iota_{\mathcal FP}(X',Y'), \fp Z) = 0$.

In summary, the space $L^2_{\rm ad}((\R^d)^T)$ provides a Lagrangian representation rich enough to capture the topological and geometric structure of adapted transport.

\subsection{Curvature of $({\rm FP}_2,\AW_2)$}
\label{sec:AOT.gen}

We start with a general result for the Wasserstein space of stochastic processes that is not restricted to Gaussian processes and could be of general interest.

To formulate it, recall that a metric space $(S,d)$ is geodesic if for every two points $x,y\in S$ there is a curve $\gamma\colon[0,1]\to S$ satisfying $\gamma(0)=x$, $\gamma(1)=y$ and 
\[d(\gamma(u),\gamma(v))=|u-v|d(x,y)\,\qquad u,v\in[0,1].\]
In the present setting, non-negative curvature is most conveniently expressed via an equivalent formulation of the triangle condition, the so-called semi-concavity inequality:
for all $x, y, z \in S$ and every constant speed geodesic  $(\gamma(u))_{u\in[0,1]}$ connecting $x$ to $y$, we have
\begin{align}
    \label{eq:semi.concavity}
d^2(\gamma(u), z) 
\geq (1 - u)d^2(x, z) + u d^2(y, z) - u(1 - u) d^2(x, y),
\end{align}
see, e.g., \cite[Proposition 12.3.3]{AmGiSa08}.
Thus, in this article, an Alexandrov space with non-negative curvature is a locally complete  geodesic space satisfying \eqref{eq:semi.concavity}.

\begin{remark}
Several closely related notions of Alexandrov spaces appear in the literature.
Our definition leans on \cite{BuGrPe92}, while other authors additionally impose metric completeness and local compactness.
In contrast, more recent references (e.g., \cite{BuBuIv01, AlKaPe24}) adopt a much less restrictive approach.
What all authors agree on is that Alexandrov spaces are metric spaces in which a curvature bound formulated via comparison properties of triangles is satisfied (whenever such triangles exist).
Consequently, restricting to locally complete geodesic spaces -- as we do here -- ensures the existence of such triangles for any triple of points in a neighborhood.
It is probably due to this terminological ambiguity that a geodesic space in which \eqref{eq:semi.concavity} holds is called a $PC$-space in \cite{AmGiSa08}; however, 
we emphasize that in our setting this notion coincides with that of an Alexandrov space with non-negative curvature.
\end{remark}

\begin{theorem} \label{thm:AW.Alex}
    $({\rm FP}_2,\mathcal{AW}_2)$ is an Alexandrov space with non-negative curvature.
\end{theorem}

The proof of Theorem \ref{thm:AW.Alex} requires the characterization of geodesic curves stated below.
To that end, for given $X,Y\in{\rm FP}_2$ we say that $X',Y' \in L^2_{\rm ad}((\R^d)^T)$  are $\mathcal{AW}_2$-optimal representatives of $X,Y$ if  $X'\sim_{\rm ad}X$ and $Y'\sim_{\rm ad} Y$ and $\AW_2^2(X,Y)=\E[|X'-Y'|^2]$.
As explained previously, $\mathcal{AW}_2$-optimal representatives correspond to optimal bicausal couplings. In particular, they always exist.

\begin{lemma}\label{lem:AW-geodesics}
    Let $X,Y\in {\rm FP}_2$ and let $(\gamma(u))_{u\in[0,1]}$ be any curve in ${\rm FP}_2$.
    Then the following are equivalent:

    \begin{enumerate}[(i)]
        \item $(\gamma(u))_{u\in[0,1]}$ is a constant speed geodesic from $X$ to $Y$;
        \item There exist $\mathcal{AW}_2$-optimal representatives $X',Y' \in L^2_{\rm ad}((\R^d)^T)$ of $X$ and $Y$ respectively, for which
        \[  \gamma(u) \sim_{\rm ad} (1-u)X' + uY', \qquad u \in [0,1].\]
    \end{enumerate}
\end{lemma}

\begin{proof}
    Let $(\gamma(u))_{u\in[0,1]}$ be an arbitrary constant speed geodesic in $({\rm FP}_2,\AW_2)$ from $X$ to $Y$.
    It was shown in \cite[Theorem 4.1]{acciaio2025absolutely} that there exist processes $Y^0,Y^1$ on some abstract filtered probability space such that 
    \[Y^u\coloneqq(1-u) Y^0 + u Y^1 \sim_{\rm ad} \gamma(u)\]
    for all $u \in [0,1]$.
    Moreover, by \cite[Proposition~6.2]{beiglbock2025brenier}, any stochastic process can be represented in $L^2_{\rm ad}((\R^d)^T)$; in particular we can find $(X^0,X^1) \in L^2_{\rm ad}( (\R^d \times \R^d)^T )$ with 
    \[(X^0,X^1) \sim_{\rm ad} (Y^0,Y^1).\]
    Hence, $X^u := (1-u) X^0 + u X^1 \sim_{\rm ad} Y^u$ for all $u \in [0,1]$, which yields the claim.

    Conversely, let $X',Y' \in L^2_{\rm ad}((\R^d)^T)$ be $\mathcal{AW}_2$-optimal representatives of $X,Y$ and set $\gamma(u)  = [(1-u)X' + uY']$ for $u\in[0,1]$.
    Then, for every $u,v\in[0,1]$,
    \begin{align}
    \label{eq:AW.geo.help}
    \AW_2(\gamma(u),\gamma(v))
    &\leq \E\left[ |(1-u)X'+uY') - ((1-v)X'+vY')|^2 \right]^{1/2}\\
    \nonumber
    &=|u-v|\E[|X'-Y'|^2]^{1/2}\\
    \nonumber
    &=|u-v|\AW_2(\gamma(0),\gamma(1)).
    \end{align}
    Moreover, by the triangle inequality, 
    \[\AW_2(\gamma(0),\gamma(1))
    \leq \AW_2(\gamma(0),\gamma(u)) + \AW_2(\gamma(u),\gamma(v)) + \AW_2(\gamma(v),\gamma(1)),\]
    from which it follows that the inequality in \eqref{eq:AW.geo.help} cannot be strict, proving the claim.
\end{proof}

\begin{proof}[Proof of Theorem \ref{thm:AW.Alex}]
    It was shown in \cite[Theorem 1.4]{BaBePa21} that $({\rm FP}_2,\mathcal{AW}_2)$ is a complete geodesic space,
    thus it remains to prove the curvature property.
    
    To this end, let $\gamma$ be a constant speed geodesic in $({\rm FP}_2,\AW_2)$ and let $X,Y \in L^2_{\rm ad}((\R^d)^T)$ be as in \Cref{lem:AW-geodesics} (ii).
    Recall that then $\AW_2^2(X,Y) = \E[|X-Y|^2]$.
    Now, let $Z \in L^2_{\rm ad}((\R^d)^T)$. Writing $X^u := (1-u)X + uY$, we find
    \begin{align*}
        |X^u - Z|^2 &= |(1-u)(X-Z)+u(Y-Z)|^2
        \\
        &= (1-u)|X-Z|^2 + u|Y-Z|^2 - u(1-u)|X-Y|^2.
    \end{align*}
    Hence, by the definition of  $\AW_2$, 
      \begin{align*}
        \AW_2^2(X^u,Z) &
        = \inf_{Z' \sim_{\rm ad} Z} \E[|X^u - Z'|^2] \\
        &= \inf_{Z' \sim_{\rm ad} Z} \left( (1-u) \E[|X-Z'|^2] + u\E[|Y-Z'|^2] - u(1-u) \AW_2^2(X,Y) \right) 
    \end{align*}
    and clearly 
     \begin{align*}
      & \inf_{Z' \sim_{\rm ad} Z} \left( (1-u) \E[|X-Z'|^2] + u\E[|Y-Z'|^2]\right)  \\
       &\quad\geq  (1-u)\inf_{Z' \sim_{\rm ad} Z} \E[|X-Z'|^2] + u\inf_{Z'' \sim_{\rm ad} Z} \E[|Y-Z''|^2] \\
       &\quad= (1-u) \AW_2^2(X,Z) + u \AW_2^2(Y,Z).
    \end{align*}
    This gives that \eqref{eq:semi.concavity} is satisfied, thus $({\rm FP}_2, \AW_2)$ has non-negative curvature.
\end{proof}

\begin{remark}
    The assertion of Theorem \ref{thm:AW.Alex} extends from $\R^d$-valued processes to processes with values in a geodesic space $(M,d)$ with non-negative curvature. 
    Moreover, the scope of Theorem \ref{thm:AW.Alex}  is likely to extend from the present discrete-time setting  to continuous time, if one considers the strong adapted Wasserstein distance (see \cite{bartl2025wasserstein})  defined using a geodesic metric on the path space.
\end{remark}

\subsection{Curvature and geodesics for Gaussian processes}

To formulate the analogue of Theorem~\ref{thm:AW.Alex} for Gaussian processes, denote by ${\rm GP}\subset {\rm FP}_2$ the set of all Gaussian processes, i.e., $\AW_2$-equivalence classes of processes $X=a+LG$ where $a\in (\R^d)^T$ and $L\in {\bf L}$, defined on a standard probability space $(\Omega,\mathcal{F},\mathbb{P})$ endowed with the filtration $(\mathcal{G}_t)_{t=1}^T$ generated by independent standard Gaussians $(G_t)_{t=1}^T$ in $\R^d$.

\begin{theorem}
\label{thm:GP.curvature}
    $({\rm GP},\mathcal{AW}_2)$  is an  Alexandrov space with non-negative curvature.
\end{theorem}

This result can be established similarly to \Cref{thm:AW.Alex}, although the proof is more involved.
The reason is that one must first establish a suitable characterization (in the sense of \Cref{lem:AW-geodesics}) of geodesics that remain entirely in the space of Gaussian processes, as this is not automatically true for all geodesics, see \Cref{rem:two.geodesics}.
On the other hand,  \Cref{thm:GP.curvature} follows directly from the Procrustes representation of $\mathcal{AW}_2$ developed in the next section and \Cref{thm:ABW.alex.intro}:

\begin{proof}[Proof of \Cref{thm:GP.curvature}]
    By definition $({\rm GP},\mathcal{AW}_2)$ and $({\bf L}/_{ \bf O},d_{\rm ABW})$ are isometric, hence the proof follows because  $({\bf L}/_{ \bf O},d_{\rm ABW})$ is an Alexandrov space with non-negative curvature by  \Cref{thm:ABW.alex.intro} (we note that the proof of \Cref{thm:ABW.alex.intro} given in Section \ref{sec:alex.background}
    is self contained and only uses \Cref{lem:AW-Gaussian-distance} from this section).
\end{proof}

\begin{remark}
\label{rem:two.geodesics}
    As it happens,  ${\rm GP}$ is not a geodesically convex subset of ${\rm FP}_2$.
    Indeed, while there always exists at least one  ${\rm FP}_2$-geodesic between Gaussian processes that stays within ${\rm GP}$, there are situations in which there are multiple geodesics, some of which leave ${\rm GP}$.
    To see that, set 
\[
L=
\begin{pmatrix}
1 & 0 \\
1 & 1
\end{pmatrix},
\qquad
M=
\begin{pmatrix}
1 & 0 \\
-1 & 1
\end{pmatrix}
\]
and put $X=LG$ and $Y=MG$. 
It turns out that $(Z^\pm(u))_{u\in[0,1]}$ given by $Z^\pm(u) = L^{\pm}(u) G$, where
\[
L^+(u) =
\begin{pmatrix}
1 & 0 \\
1-2u & 1
\end{pmatrix},
\,\quad
L^-(u) =
\begin{pmatrix}
1 -2u & 0 \\
1 & 1
\end{pmatrix},
\]
are both constant speed geodesics between $L$ and $M$.

Moreover, let $\xi$ be a fair coin-flip that is independent of $(G_t)_{t=1}^T$ and set 
\[ Z(u):=\begin{cases}
    Z^+(u), &\xi = \text{head}\\
    Z^-(u), & \xi = \text{tail}.
\end{cases}
\]
Then $(Z(u))_{u\in[0,1]}$ can be checked to be a geodesic, but for every $t=1,2$ and $u\in(0,1)$, $Z_t(u)$ is not Gaussian (it is a mixture of Gaussians).
\end{remark}

\begin{remark}
\label{rem:full.rank.not.geo}
    A minor modification of the processes considered in the previous example gives full rank matrices $L$ and $M$ for which there exists a unique ${\rm FP}_2$-geodesic, and that geodesic has a degenerate matrix at $u=\frac{1}{2}$.
    In particular, the space of all Gaussian processes having full-rank covariance matrices is not a geodesically convex subset of ${\rm FP}_2$; nor of ${\rm GP}$ for that matter (in contrast to the $\W_2$-setting where full rank Gaussians are geodesically convex). 
    Refer also to Example 4.2 in \cite{acciaio2025entropic}.
\end{remark}

Finally, we require the following (immediate) extension of the closed form formula for the $\mathcal{AW}_2$-distance between full rank Gaussians established in  \cite{acciaio2025entropic} to the present setting:

\begin{lemma} \label{lem:AW-Gaussian-distance}
    For every $L,M\in{\bf L}$ and $a,b\in(\R^d)^T$, setting $X:=a+LG$ and $Y:=b+MG$, we have
     \begin{align*}\mathcal{AW}_2^2(X,Y) 
     &=\|a-b\|_2^2 + \|L\|_{\rm F}^2 + \|M\|_{\rm F}^2 - 2{\rm tr}(S),
     \end{align*}
     where $S={\rm diag}(S_1,\dots,S_T)$ and each $S_t$ is the diagonal matrix containing the singular values of $(M^\top L)_{t,t}$ for $t = 1, \dots, T$.
\end{lemma}
\begin{proof}
    For full-rank $L,M\in{\bf L}$ the statement is proven in \cite{acciaio2025entropic}.
    The general case follows from an approximation: let $L^n,M^n\in{\bf L}$ be full-rank matrices  such that $L^n\to L$ and $M^n\to M$ as $n\to\infty$, and set
    \[X^n:=a+L^nG \quad\text{and}\quad Y^n:=b+M^nG.\]
    Since the sequences converge in $L^2$, $\mathcal{AW}_2(X^n,X)\to 0$ and  $\mathcal{AW}_2(Y^n,Y)\to 0$ as $n\to\infty$.
Therefore,
\begin{align*}
     \mathcal{AW}_2^2(X,Y)
&=\lim_{n\to\infty} \mathcal{AW}_2^2(X^n,Y^n) \\
&=\lim_{n\to\infty} \left( \|a-b\|_2^2 + \|L^n\|_{\rm F}^2 + \|M^n\|_{\rm F}^2 - 2{\rm tr}(S^n) \right),
\end{align*}
where $S^n={\rm diag}(S_1^n,\dots,S_T^n)$ and each $S_t^n$ is the diagonal matrix containing the singular values of $((M^n)^\top L^n)_{t,t}$.
Since $(M^n)^\top L^n\to M^\top L$ as $n\to\infty$, $S^n\to S$ and the claim follows.
\end{proof}

We describe $\mathcal{AW}_2$--optimal representatives for Gaussian processes and 
moreover characterize geodesics in ${\rm GP}$.

\begin{theorem}
\label{thm:AW.curves}
Let $X,Y \in {\rm GP}$ be represented by $L,M\in {\bf L}$, respectively.
\begin{enumerate}[(i)]
\item   If $P\in {\bf O}^\ast(L,M)$, then
    \[ (X',Y') := (LG,MPG) \]
    are $\mathcal{AW}_2$-optimal representatives of $X,Y$.
    \item If $P\in {\bf O}^\ast(L,M)$, then 
    \begin{align} 
    \label{eq:AW.GP.CSG}
    \gamma(u):= [((1-u)L + u MP)G], \quad u\in[0,1]
    \end{align}
    defines a constant-speed geodesic in ${\rm GP}$ with $\gamma(0) \sim_{\rm ad} X$ and $\gamma(1) \sim_{\rm ad} Y$.
    \item If $(\gamma(u))_{u\in[0,1]}$ is a constant-speed geodesic in ${\rm GP}$ with $\gamma(0) \sim_{\rm ad} X$ and $\gamma(1) \sim_{\rm ad} Y$, then there exists $P\in {\bf O}^\ast(L,M)$ such that \eqref{eq:AW.GP.CSG} holds.
\end{enumerate}
\end{theorem}
\begin{proof}
    Let $P\in {\bf O}^\ast(L,M)$. By \Cref{lem:ABW-distance,} and the definition of ${\bf O}^\ast(L,M)$,
    \[ \mathcal{AW}_2(X,Y) 
    =  \|L-MP\|_{\rm F}
    = \E \left[ |LG - MPG|^2\right]^{1/2},\]
    where the second equality follows from a straightforward computation.
    Since $LG\sim_{\rm ad}X$ and $MPG\sim_{\rm ad}Y$, the first claim follows.
    
    The second claim follows from (i) and \Cref{lem:AW-geodesics}.

    As for the third claim, assume that $(\gamma(u))_{u\in[0,1]}$ is a constant-speed geodesic in ${\rm GP}$.
    Then $(\gamma(u))_{u\in[0,1]}$ can be identified with a constant-speed geodesic $(\tilde\gamma(u))_{u\in[0,1]}$ in ${\bf L}/_{\bf O}$ and \Cref{thm:identification-of-geodesics} implies that there is $P\in {\bf O}^\ast(L,M)$ satisfying $\tilde\gamma(u) = [(1-u)L+uMP]$ for all $u\in[0,1]$.

    Since $\tilde\gamma(u)$ corresponds to $[((1-u)L+uMP)G]$ the claim follows.
\end{proof}

\begin{remark}
    It is crucial to restrict in (iii) to constant-speed geodesic in ${\rm GP}$ rather than ${\rm FP}_2$.
    Indeed, if $|{\bf O}^\ast(L,M)|>1$ one can `randomize' displacement interpolations and construct a constant speed geodesic $(\gamma(u))_{u\in[0,1]}$ in ${\rm FP}_2$ for which $\gamma(u)\notin {\rm GP}$ for all $u\in(0,1)$, see \Cref{rem:two.geodesics}.
\end{remark}

\section{Fine geometric structure}
We now turn to the proofs of the results announced in the introduction, 
beginning with the case of general $L \in {\bf L}$ and specializing later to regular $L \in {\bf L}^{\rm reg}$.

Before we start, let us comment on frequently used notation. 
We denote by $\mathrm{O}(d)$ the set of all $d \times d$-dimensional orthogonal matrices.
For a matrix $A \in \mathbb{R}^{m \times m}$, we write $\sigma_i(A)$ for its $i$-th singular value and 
its singular value decomposition (SVD in what follows) as
$A = USV^\top$; 
here $U,V\in \mathrm O(m)$ and $S={\rm diag}(\sigma_1(A), \dots, \sigma_m(A))$ 
with  the convention that $\sigma_1(A) \geq \ldots \geq \sigma_m(A)$.

For a matrix $A\in \mathbb{R}^{dT \times dT}$ we will also use the block-notation $A=(A_{s,t})_{s,t=1}^T$ with 
$A_{s,t} \in \mathbb{R}^{d\times d}$, $s,t = 1, \dots, T$ and denote the $t$-th block row and $t$-th block column of $A$ by  
\begin{align*}
    A_{t,\cbullet}& = (A_{t,1}\; A_{t,2}\; \cdots\; A_{t,T}) \in \mathbb{R}^{d \times dT} \text{ and}\\
    A_{\cbullet,t} &= (A_{1,t}, A_{2,t}, \dots, A_{T,t})^\top \in \mathbb{R}^{dT \times d}, 
\end{align*}
respectively.
In particular, for $A,B\in \mathbb{R}^{dT \times dT}$,  $(AB)_{t,t}= A_{t,\cbullet}B_{\cbullet,t}$.
Moreover, if $D\in{\bf L}$ is block-diagonal, then $(AD)_{t,t}= A_{t,t} D_t$.

Throughout this article, when we say that a matrix  $A\in\R^{m\times m}$ is positive semi-definite, we implicitly assume that it is symmetric.

\subsection{Basic properties}

We will start with the following simple but crucial observation.
Recall the set
\[
\mathbf O = \left\{\diag(O_1, \dots, O_T): O_t \in \mathrm O(d), t = 1, \dots, T \right\}.
\]
The Procrustes representation of $\mathcal{AW}_2$ obtained in the next lemma plays a central role throughout the paper.
\begin{lemma} \label{lem:ABW-distance}
    For every $L, M \in \mathbf L$,
    \[
    d_{\rm ABW}(L,M)
    = \inf_{O\in {\bf O}} \|L  - MO\|_{\rm F}.\]
Moreover, the infimum is attained.
\end{lemma}
\begin{proof}
    An application of the bilinearity of the Frobenius inner product gives
    \begin{align}
         \nonumber\inf_{O \in \bf O } \| L - MO\|_{\rm F}^2
        &=  \| L \|_{\rm F}^2  + \| M \|_{\rm F}^2   - 2 \sup_{O \in \bf O } \langle L , MO\rangle_{\rm F}\\
        \label{eq:sup.equal.sum.sup}
        &=  \| L \|_{\rm F}^2  + \| M \|_{\rm F}^2- \sum_{t = 1}^T \sup_{O_t \in {\rm O}(d)}\tr\left((M^\top L)_{t,t}O_t^\top \right).
    \end{align}
It follows from the Von Neumann trace inequality that 
\begin{equation} \label{eq:trace-optimizer}
\sup_{O \in \mathrm{O}(d)} \tr\left((M^\top L)_{t,t}O^\top \right) = \tr(S_t),
\end{equation}
where $S_t$ is the diagonal matrix containing the singular values of $(M^\top L)_{t,t}$.
Hence the first statement follows from \Cref{lem:AW-Gaussian-distance} and the definition of $d_{\rm ABW}$.

The second statement follows from the compactness of ${\bf O}$.
\end{proof}

Throughout, we make repeated use of the representation
 \begin{align}
 \label{eq:rep.d.prok}
 \begin{split}
         d_{\rm ABW}^2(L,M)
         &=\inf_{O \in \bf O } \| L - MO\|_{\rm F}^2
        =  \| L \|_{\rm F}^2  + \| M \|_{\rm F}^2   - 2 \sup_{O \in \bf O } \langle L , MO\rangle_{\rm F}
        \end{split}
    \end{align}
for $L,M\in\bf L$, omitting explicit references to \Cref{lem:ABW-distance}.

The next lemma, while elementary, allows for a characterization of the optimizers of \eqref{eq:rep.d.prok}.
To state it, for  $r \in \{0, \dots, d\}$, set
\[
\mathbf Q(r):=  \left \{ \begin{pmatrix}
    {\rm Id} & 0 \\ 0 & Q'
\end{pmatrix} : Q' \in \mathrm O (d -r)\right\} \subseteq \mathrm O (d).
\]

\begin{lemma} 
\label{prp:optimizer-characterization}
    Let $A \in \mathbb R^{d \times d}$ with SVD $A = U S V^\top$
    and let $O \in \mathrm O(d)$. 
    Then the following are equivalent:
    \begin{enumerate}[(i)]
        \item $AO$ is positive semi-definite;
        \item $\tr(AO) = \tr(S)$;
        \item $O = VQU^\top$ for $Q \in \mathbf Q(\operatorname{rk}(A))$. 
    \end{enumerate}
\end{lemma}
We provide the proof of \Cref{prp:optimizer-characterization} for completeness.
\begin{proof}
We start by proving that (i) implies (ii). 
Indeed, if $AO$ is positive semi-definite, 
its SVD is of the form $AO = W \tilde S W^\top$ and clearly $\tr(AO) = \tr(\tilde S)$.
In particular, $A = W \tilde S (O W)^\top$ is an SVD of $A$, 
thus it is apparent that $A$ and $AO$ have the same singular values and $\tr(\tilde S)=\tr(S)$.

To show that (ii) implies (iii), 
let $O \in \mathrm O(d)$ such that $\tr(AO) = \tr(S)$.
Consider $Q:= V^\top O U \in \mathrm O(d)$ and note that
\[
\tr(AO) 
= \tr(USV^\top O) 
= \tr(S V^\top OU)
= \tr(SQ)
= \sum_{i = 1}^d \sigma_{i}(A)Q_{ii}
=:(\ast),
\]
where in the last equality we used the fact that $S$ is a diagonal matrix.
Since $Q_{ii} \leq 1$ by  orthogonality,  $(\ast) \leq \tr(S)$ with equality if and only if $Q_{ii} = 1$ 
for all $i = 1, \dots, \operatorname{rk}(A)$.

It remains to show that (iii) implies (i). 
To this end, let $O = VQU^\top$ for $Q \in \mathbf Q(\operatorname{rk}(A))$. 
Then $SQ = QS = S$ and thus
\[AO = USV^\top V Q U^\top = U S U^\top,\]
 which is clearly positive semi-definite. 
\end{proof}

A key ingredient in the (geometric) analysis of $d_{\rm ABW}$ is the
characterization of those elements in $O\in{\bf O}$ that are optimal in \eqref{eq:rep.d.prok}.

\begin{lemma} 
\label{lem:ABW-optimizers}
\label{special-optimizers}
 Let $L,M \in \mathbf L$ and $P\in {\bf O}$.
 Then the following are equivalent: 
    \begin{enumerate}[(i)]
        \item $P \in \mathbf O^*(L,M)$, i.e.\ $\|L-MP\|_{\rm F} = d_{\rm ABW}(L,M)$;
        \item $(M^\top L)_{t,t}P^\top_t$ and $P^\top_t(M^\top L)_{t,t}$ are positive semi-definite for all $t=1,\dots,T$;
        \item It holds that
\begin{equation} \label{eq:O(L,M)}
    P\in 
    \left\{
    \diag(P_1,\dots,P_T) :
    \begin{array}{l}
    P_t = U_t Q_t V_t^\top
    \text{ where } (M^\top L)_{t,t} = U_t S_t V_t^\top 
    \\ 
    \text{denotes an SVD and }
    Q_t \in \mathbf Q(\operatorname{rk}(S_t))
    \end{array}
    \right\};
    \end{equation}
        \item $\tr((M^\top L)_{t,t} P^\top_t)  
        = \tr(S_t)$ for all $t = 1, \dots, T$, where $S_t$ is the diagonal matrix containing the singular values of $(M^\top L)_{t,t}$.
    \end{enumerate}   
    In particular, ${\rm Id} \in \mathbf O^*(L,M)$ if and only if
$(M^\top L)_{t,t}$ is positive semi-definite for all $t = 1, \dots, T$.
\end{lemma}
\begin{proof}
First note that the equivalence of (i) and (iv) follows from \eqref{eq:sup.equal.sum.sup} and \eqref{eq:trace-optimizer}. 
Furthermore, applying \Cref{prp:optimizer-characterization} (with $A=(M^\top L)_{t,t}$ and $O=P^\top$) gives the equivalence of (ii), (iii) and (iv), 
where we additionally note that a matrix $A$ is positive semi-definite if and only if $O^\top A O$ is positive semi-definite for every (block-) orthogonal $O$. 
\end{proof}
Finally, we collect a few simple yet important facts about ${\bf O}^\ast(L,M)$ that will be used repeatedly in the remainder of this section.

\begin{lemma} \label{lem:O^*-facts}
For every $L,M \in \mathbf L$, the following hold:
\begin{enumerate}[(i)]
    \item
    $O \in \mathbf O^*(L,L)$ if and only if $O \in \mathbf O$ and $LO = L$.
    In particular, ${\rm Id} \in \mathbf O^*(L,L) $;
    \item
    If $O \in \mathbf O^*(L,L)$ and $P \in \mathbf O^*(L,M)$, then $PO \in \mathbf O^*(L,M)$;
    \item
    If $P \in \mathbf O^*(L,M)$, then $P^\top\in \mathbf O^*(M,L)$;
    \item
    If $(M^\top L)_{t,t}$ is symmetric for every $t = 1, \dots, T$, then $\mathbf O^*(L,M) =\bO^*(M,L)$.
    \end{enumerate}
\end{lemma}
\begin{proof}
To prove  (i), note that clearly  $O \in \mathbf O^*(L,L) $
    if and only if $O \in \mathbf O$ and 
    $0 
    = \inf_{P \in \mathbf O} \|L-L P \|_{\rm F}
    = \|L-LO \|_{\rm F},
    $
    which is equivalent to $LO=L$.

As for (ii), we have that
    \[ \inf_{\tilde P \in \mathbf O} \|L-M \tilde P \|_{\rm F} 
    = \| L - MP\|_{\rm F} 
	=\| (L - MP)O\|_{\rm F} 
	=\| L - MPO\|_{\rm F}\]
    and since  $PO\in{\bf O}$, the claim follows from the definition of ${\bf O}^\ast(L,M)$.

Finally, (iii) follows from 
\[ \| L - MP\|_{\rm F} 
= \| (L - MP)P^\top\|_{\rm F} 
= \| M - LP^\top\|_{\rm F},  \]
and (iv) is a consequence of (iii).
\end{proof}

\subsection{Properties of geodesics}

A key step in our analysis concerns understanding the structure of the set 
\begin{equation} \label{eq:V-def}
\mathbf V(L) := \left\{MP -L : M  \in \mathbf L, P \in \mathbf O^*(L,M) \right\}, \qquad L \in \mathbf L,
\end{equation}
which, as we shall see, corresponds to the geodesic curves starting from $L$.
\begin{proposition}
\label{prop:O.equal.O}
    Let $L \in \mathbf L$ and $V \in \mathbf V(L)$. 
    Then, for all $u \in [0,1)$,
    \[
     \mathbf O^*(L,L)= \mathbf O^*(L,L+uV).
    \]
    Moreover,  $\mathbf O^*(L,L) \subseteq \mathbf O^*(L,L+V)$.
\end{proposition}

\begin{proof}
    The proof in case that $u\in[0,1)$ follows from   \Cref{lem:O.equal.O.1} and \Cref{lem:O.equal.O.2}.
    Moreover, since $((L+V)^\top L)_{t,t}$ is positive semi-definite for all $t=1,\dots, T$ by \Cref{lem:O.equal.O.1}, \Cref{special-optimizers} implies that ${\rm Id} \in \mathbf O^*(L,L+V)$, thus
    $\mathbf O^*(L,L) \subseteq  \mathbf O^*(L,L+V)$ by  \Cref{lem:O^*-facts} (ii).
\end{proof}

From the above result, we immediately get:
\begin{corollary}
\label{lem:distance.for.V(L)}
    For every $L\in {\bf L}$, $V\in {\bf V}(L)$, and $u\in[0,1]$,   $d_{\rm ABW}(L, L+ uV) = u\|V\|_{\rm F}$.
\end{corollary}

The characterization of  ${\bf O}^\ast(L,M)$ in \Cref{special-optimizers} relies on the rank of the involved matrices. 
This becomes crucial in what follows, 
and it is useful to note the following. 

\begin{remark} \label{rmk:rank-im-ker}
Let $L,M \in \mathbf L$ and $t\in\{1,\dots,T\}$.
    The following are equivalent:
    \begin{enumerate}[(i)]
        \item $\operatorname{rk}((L^\top L)_{t,t}) = \operatorname{rk}((M^\top L)_{t,t})$;
        \item $\operatorname{img}((L^\top L)_{t,t}) = \operatorname{img}((M^\top L)_{t,t}^\top)$;
        \item $\operatorname{ker}((L^\top L)_{t,t}) = \operatorname{ker}((M^\top L)_{t,t})$.
    \end{enumerate}
This follows from the rank-nullity theorem and since
\[
\operatorname{img}\left((L^\top L)_{t,t}\right)^\perp
=\ker\left((L^\top L)_{t,t}\right)
\subseteq \ker\left((M^\top L)_{t,t}\right)
= \operatorname{img}\left((M^\top L)_{t,t}^\top\right)^\perp.
\]
\end{remark}
\begin{lemma} \label{lem:O.equal.O.1}
    Let $L \in \mathbf L$, $V \in \mathbf V(L)$ and $t\in\{1,\dots,T\}$.
    Then, for every $u \in [0,1]$, 
    the matrix $( ( L + uV ) ^\top L )_{t,t}$ is positive semi-definite, and for every $u \in [0,1)$,
    \[
    \operatorname{rk}\left( ( (L+uV)^\top L )_{t,t}\right) 
    = \operatorname{rk}\left(( L^\top L)_{t,t} \right).
    \]  
\end{lemma}

\begin{proof} 
    We have $V = MP - L$ for $M \in \mathbf L$ and $P\in \mathbf O^*(L,M)$, hence
    \begin{align}
        \label{eq:geodesic-expanded}
    ( ( L + uV ) ^\top L )_{t,t}
    = (1-u) (L^\top L)_{t,t} + u P_t^\top (M^\top L)_{t,t}.
    \end{align}
    Since $ P_t^\top (M^\top L)_{t,t}$ is positive semi-definite by \Cref{lem:ABW-optimizers}, 
     $((L + uV)^\top L )_{t,t}$ is positive semi-definite for all $u \in [0,1]$.

    For the second assertion, by \Cref{rmk:rank-im-ker}, it suffices to show that $  \operatorname{ker}( ( (L+uV)^\top L )_{t,t})     = \operatorname{ker}(( L^\top L)_{t,t} )$ for all $t=1,\dots,T$.
    To that end, observe that 
    $\ker((L^\top L)_{{t,t}}) = \ker(L_{\cbullet, t})$
    and $P_t^\top(M^\top L)_{t,t} =P^\top_t M_{\cbullet,t}^\top L_{\cbullet, t}$.
    Thus, for every $x \in \ker((L^\top L)_{{t,t}})$, 
    we have that $((L + uV) ^\top L)_{t,t}x = 0$, 
    hence \[\ker((L^\top L)_{{t,t}}) \subseteq \operatorname{ker}((( L + uV ) ^\top L )_{t,t})\]
    for all $u \in [0,1]$. 
    On the other hand, from \eqref{eq:geodesic-expanded} and positive semi-definiteness, we have for all $u\in[0,1)$ that
    \[\ker((( L + uV ) ^\top L )_{t,t})
    \subseteq \ker((L^\top L)_{{t,t}}).\qedhere\]  
\end{proof}

\begin{lemma}
\label{lem:O.equal.O.2}
    For every $L, M \in \mathbf{L}$,     the following are equivalent:
    \begin{enumerate}[(i)]
        \item $(M^\top L)_{t,t}$ is positive semi-definite and $\operatorname{rk}((L^\top L)_{t,t}) = \operatorname{rk}((M^\top L)_{t,t})$ 
    for all $t = 1, \dots, T$;
        \item $\mathbf O^*(L,M) = \mathbf O^*(L,L)$.
    \end{enumerate}
\end{lemma}
\begin{proof}
    For $ t = 1, \dots, T$, set
    \[ r_t := \operatorname{rk}\left((M^\top L)_{t,t}\right)
    \quad \text{and}\quad
    \tilde r_t := \operatorname{rk}\left((L^\top L)_{t,t}\right).
    \] 
    We start by showing that (i) implies (ii). 
    Since $(M^\top L)_{t,t}$ is positive semi-definite, $ {\rm Id }\in \mathbf O^*(L,M)$ by \Cref{special-optimizers} and therefore $\mathbf O^*(L,L) \subseteq \mathbf O^*(L,M)$ by \Cref{lem:O^*-facts} (ii).
    To prove the reverse inclusion, let
    \[(M^\top L)_{t,t} = V_t S_t V_t^\top
    \quad\text{and}\quad
    (L^\top L)_{t,t} = W_t \tilde S_t W_t^\top
    \]denote respective SVDs
    and consider $P = \diag(P_1, \dots, P_T)  \in \mathbf O^*(L,M)$. 
    Then, using \eqref{eq:O(L,M)}, each $P_t$ is of the form $P_t  = V_t Q_t V_t^\top$ for some $Q_t \in \mathbf Q(r_t)$. 
    Now note that $r_t = \tilde r_t$, hence our claim follows if we can find $\tilde Q_t \in \mathbf Q(r_t)$ 
    such that $P_t  = W_t \tilde Q_t W_t^\top$.
    
    To this end, consider the representation 
    \[ V_t = (V' , V'')
    \quad\text{and}\quad
    W_t = (W' , W'')\]
    for $V' \in \mathbb R^{d \times r_t}$ 
    and $V'' \in \mathbb R^{d \times (d-r_t)}$ and, similarly, $W'\in \mathbb R^{d \times r_t}$ 
    and $W'' \in \mathbb R^{d \times (d-r_t)}$. 
    
    By \Cref{rmk:rank-im-ker} and using that $( M^\top L )_{t,t}=( M^\top L )_{t,t}^\top$, we have 
    $ \operatorname{img}( ( L^\top L )_{t,t} ) = \operatorname{img}( ( M^\top L )_{t,t})$, 
    hence the columns of $V'$ and $W'$ both form an orthonormal basis of the same image space. 
    Consequently, the columns of $V''$ and $W''$ form orthonormal bases of its orthogonal complement.
    Thus for any column $w_j$ of $W'$ there exists a vector $u_j \in \mathbb R^{r_t}$  such that $w_j = V' u_j$. 
    Hence $W'= V'U'$ for   $U' := (u_1,\dots, u_{r_t}) \in \mathbb R^{r_t \times r_t}$.
    Analogously, we find $U'' \in \mathbb R^{(d-r_t) \times (d-r_t)}$ such that $W'' = V'' U''$.
    Next, note that we have 
    $(W')^\top W' = (V')^\top V' = {\rm Id}$ due to orthogonality of the columns.     
    Therefore, 
    \[
    {\rm Id} = (W')^\top W' = (U')^\top (V')^\top V'U'
    = (U')^\top U',
    \]
    which shows that $U'$ is orthogonal. 
    In a similar fashion,  $U''$ is orthogonal. 
    In particular, setting 
    \[U := \begin{pmatrix}
        U'  & 0  \\ 0  & U''
    \end{pmatrix} \in \mathrm O(d),
    \]
    we have $W_t = V_tU$ and furthermore
    \[
    P_t = V_t Q_t V_t^\top 
    = V_t U U^\top Q_t U U^\top V_t^\top
    = W_t U^\top Q_t U W_t.
    \]
    Now clearly $\tilde Q := U^\top Q_t U$ is again in $\mathbf Q(r_t)$ which concludes the argument.

To show that (ii) implies (i), note that it is clear from \Cref{special-optimizers} 
that $( M^\top L )_{t,t}$ is positive semi-definite for all $t = 1, \dots, T$. 
Thus we can use arguments similar to those above to show that 
$\mathbf O^\ast(L,M) = \mathbf O^\ast(L,L)$
if and only if  for all $t = 1, \dots, T$,
\[
V_t \mathbf Q(r_t)V_t^\top = W_t \mathbf Q(\tilde r_t)W_t^\top ,
\]
which is equivalent to  $r_t = \tilde r_t$ for all $t = 1, \dots, T$.
\end{proof}
An immediate consequence of \Cref{lem:O.equal.O.2} is the following convenient characterization of optimizers in the full rank case. 
\begin{corollary} \label{cor:O=Id}
For every $L,M\in{\bf L}$, $(M^\top L)_{t,t}$   is positive definite for all $ t = 1, \dots T $ if and and only if $ \mathbf O^*(L,M) = \{\rm Id\}$. 
\end{corollary}

\subsection{Geodesics and displacement interpolations}
For $L\in{\bf L}$, denote by
\[\Gamma ([L]) \coloneq 
    \left\{
    \gamma :[0,1] \rightarrow \mathbf L /_{\mathbf O}  : \gamma \text{ is a constant speed geodesic with } \gamma(0) = [L]
    \right\}
\]
the set of all constant speed geodesics emerging from $[L]$.
Thus, for every  $\gamma \in \Gamma ([L])$, 
there exist matrices $L(u) \in {\bf L}$ such that $\gamma(u) = [L(u)]$ for all $u \in [0,1]$
and $[L(0)] = [L]$. 
One of the essential steps in this paper is to show that, in fact, there is a one-to-one correspondence between constant speed geodesics and displacement interpolations built with elements in $\mathbf V(L)$ -- that is, there exists an element $V \in \mathbf V(L)$ such that
$[L(u)] = [L+uV]$ for every $u \in [0,1]$. 

We start by observing that every such displacement interpolation is in fact a constant speed geodesic.

\begin{lemma}
    \label{lem:geodesics.one.side}
Let $V\in {\bf V}(L)$ and set $\gamma(u):=[L+uV]$ for $u\in[0,1]$.
Then $\gamma \in \Gamma([L])$.
\end{lemma}
The proof follows from similar arguments as the proof presented for \Cref{lem:AW-geodesics}, we provide it for completeness.
\begin{proof}
For every $u,v\in[0,1]$,
\begin{align*} 
d_{\rm ABW}(\gamma(u),\gamma(v) )
&=\inf_{O\in {\bf O}} \| (L+uV) - (L+vV)O\|_{\rm F} \\
&\leq |u-v|\|V\|_{\rm F} 
=|u-v|d_{\rm ABW}(\gamma(0),\gamma(1)),
\end{align*}
where the last equality is due to \Cref{lem:distance.for.V(L)}.
Therefore, by the triangle inequality, for every $u,v\in[0,1]$ with $u\leq v$, 
\begin{align*} 
d_{\rm ABW}(\gamma(0),\gamma(1) )
&\leq d_{\rm ABW}(\gamma(0),\gamma(u) ) + d_{\rm ABW}(\gamma(u),\gamma(v) ) + d_{\rm ABW}(\gamma(v),\gamma(1) ) \\
&\leq d_{\rm ABW}(\gamma(0),\gamma(1) )
\end{align*}
and thus $d_{\rm ABW}(\gamma(u),\gamma(v) )=|u-v| d_{\rm ABW}(\gamma(0),\gamma(1))$.
\end{proof}
However, the same geodesic $\gamma \in \Gamma([L])$ may arise from multiple choices of $V \in \mathbf V(L)$.

\begin{example}
Let $d = 1$, $T = 2$, consider 
$L = \begin{psmallmatrix} 1 & 0\\ 0 & 0 \end{psmallmatrix}$
and $M = -{\rm Id}$. Then
\[
\mathbf O^*(L,M) = \left \{\begin{pmatrix} -1 & 0 \\ 0 & \xi \end{pmatrix} : \xi=\pm 1 \right\} \]
and in particular
\begin{align*}
    \{ MP-L : P\in \mathbf O^*(L,M)\}
    &= \left\{  \begin{pmatrix} 0 & 0 \\ 0 & \xi \end{pmatrix} : \xi=\pm 1\right\}=:\{ V^+,V^-\}.
\end{align*}
Therefore, $L^\pm (u) := L + u V^{\pm} $  for $u\in[0,1]$ define displacement interpolations between $[L]$ and $[M]$ which satisfy $[L^+(u)]= [L^-(u)]$ for all $u\in[0,1]$.

Observe that for the matrices  $V^\pm$ given we have that  $V^+\sim_{{\bf O}^\ast(L,L)} V^-$, that is, $V^+=V^-O$ for some 
\[ O\in 
\mathbf O^*(L,L ) = \left \{\begin{pmatrix} 1 & 0 \\ 0 & \xi \end{pmatrix} : \xi=\pm 1 \right\} .\]
Recalling that 
\[ d_{{\bf O}^\ast(L,L)}(V,W) = \inf_{O \in {\bf O}^\ast(L,L)} \|V-WO\|_{\rm F},\]
it follows that  $d_{{\bf O}^\ast(L,L)}(V^+,V^-)=0$.
\end{example}

As the previous example suggests, in order to obtain a one-to-one correspondence between $\Gamma([L])$ and  ${\bf V}(L)$, 
one needs to factorize by ${\bf O}^\ast(L,L)$; the orthogonal transformations that leave $L$ invariant.

Taking the quotient by this action, that is, considering ${\bf V}(L)/_{{\bf O}^\ast(L,L)}$, indeed yields a set isomorphic to $\Gamma(L)$:

\begin{theorem}
\label{thm:identification-of-geodesics}
    For every $L \in \mathbf L$, the spaces $\Gamma ([L])$ and $ {\bf V} (L)/_{{\bf O}^\ast(L,L)}$ are isomorphic.
    In particular, every $\gamma\in\Gamma([L])$ can be represented as
    \[ \gamma(u) = [ L+uV], \qquad u\in[0,1]  \]
    for some $V\in{\bf V} (L)$. Moreover,  if  $L\in {\bf L}^{\rm reg}$, then $V$ is unique.
\end{theorem}

\begin{remark}
    It is straightforward to show that ${\mathbf O^*(L,L)}$ is a group (hence $d_{\mathbf O^*(L,L)}$ is a pseudo-metric) and thus the quotient space  ${\bf V} (L)/_{{\bf O}^\ast(L,L)}$  is well-defined.

    Moreover, $\mathbf V(L)$ is compatible with the factorization of ${\bf L}$ that we use here:    for every $L$ and $O \in \mathbf O$, 
     $\mathbf V(LO) = \mathbf V(L) \cdot O$.
 \end{remark}

In the proof of the above theorem, we will employ the following simple observation:
Let $A,B,C \in \mathbf L$ satisfying $\|A -C\|_{\rm F} = \|A - B\|_{\rm F} + \|B - C\|_{\rm F}$.
Since the Frobenius norm is strictly convex, this equality implies that $A - B$ and $B - C$ are both multiples of $A - C$ by a non-negative scalar.
It follows that if
\[
    u \|A - C\|_{\rm F} =  \|A - B\|_{\rm F} \text{ and } (1-u) \|A - C\|_{\rm F} = \|B - C\|_{\rm F}
\]
for some $u \in (0,1)$, then $B = (1-u) A + u C$.
\begin{proof}[Proof of \Cref{thm:identification-of-geodesics}]
    Let $V,W \in {\bf V}(L)$ such that $V=WO$ for some $O\in {\bf O}^\ast(L,L)$.
    Set $\gamma(u) :=  [L+uV]$ and $\eta(u) :=  [L+uW ]$ for $u\in[0,1]$, then $\gamma,\eta \in\Gamma([L])$ by Lemma~\ref{lem:geodesics.one.side}.
    Note that
    \begin{align*} 
    d_{\rm ABW}(\eta(u),\gamma(u))
    &=\inf_{P\in \bf O} \| (L+uW) - (L+uV)P\|_{\rm F} \\
    &\leq \| (L+uVO) - (L+uV)O\|_{\rm F}
    =0,
    \end{align*}
    where we used that $LO=L$ because $O\in {\bf O}^\ast(L,L)$.
    Thus $\gamma = \eta$ and therefore  ${\bf V}(L)/_{{\bf O}^\ast(L,L)}$ can be identified as a subset of $ \Gamma([L])$.

    Conversly, let $\gamma\in\Gamma([L])$ and fix $v \in [0,\frac{1}{2}]$. 
    Since $\gamma$ is a geodesic, one can verify that there are $A' \in \gamma(v)$, $B,B' \in \gamma(\tfrac{1}{2})$, and $C,C' \in \gamma(1)$ with the property that
    \begin{align}
        \label{eq:frobenious.strict.convec}
        \begin{split}
            d_{\rm ABW}(\gamma(0),\gamma(1)) &= \|L - C\|_{\rm F} = \|L - B\|_{\rm F} + \|B -C\|_{\rm F},
        \\
        &= \|L - C'\|_{\rm F} = \|L - A'\|_{\rm F} + \|A' -B'\|_{\rm F} + \|B' - C'\|_{\rm F}.
        \end{split}
    \end{align}
     (Indeed, pick any $\tilde C \in \gamma(1)$ and put $C=\tilde C P  \in \gamma(1)$ for some $P\in {\bf O}^\ast(L, \tilde C)$. 
     Then continue,  constructing $B$ using this $C$, etc.) 
    From the first line in \eqref{eq:frobenious.strict.convec}, we derive $B = \frac{1}{2}(L+C)$.
    Similarly, the second line yields
    \[
       \|L - B'\|_{\rm F} + \|B'-C'\|_{\rm F} = \|L - C'\|_{\rm F} = \|L - A'\|_{\rm F} + \|A' - C'||_{\rm F},
    \]
    from where we derive $A' = (1-v)L + vC'$ and $B' = \frac{1}{2}(L+C')$.
    Furthermore, because
    \[
        \|L - B\|_{\rm F} = d_{\rm ABW}\left(\gamma(0),\gamma\left(\tfrac{1}{2}\right)\right) = \|L - B'\|_{\rm F},
    \]
    there exists $P \in \mathbf O^\ast(L,B)$ with $BP = B'$.
    As $C-L \in \mathbf V(L)$,
    we find by \Cref{prop:O.equal.O} that $\mathbf O^\ast(L,L) = \mathbf O^\ast(L,L + \frac12 (C-L)) = \mathbf O^\ast(L,B)$,
    and deduce that
    \[
       \frac{1}{2}(L+C') = B' = BP = \frac{1}{2}(L+CP),
    \]
    that is, $C' = CP$.
    In particular, 
    \[
        A' = (1 - v)L + vC' = ( (1-v)L + vC)P,
    \]
    which yields that $\gamma(v) = [(1-v)L + vC]$.
    By analogous reasoning for $v \in (\frac{1}{2},1)$, we find that $\gamma(v) = [(1-v)L + vC]$ for any $v \in [0,1]$.
    Hence we can identify $\Gamma([L])$ as subset of ${\bf V}(L)$ and thus with a subset of ${\bf V}(L)/_{{\bf O}^\ast(L,L)}$ (by the same reasoning as in the first step).
    
    The last statement follows by \Cref{cor:O=Id}: ${\bf O}^\ast(L,L)=\{ {\rm Id}\}$ if $L$ is regular.
\end{proof}

\begin{corollary}
\label{cor:unique.geo}
    For $L,M\in {\bf L}$, set 
    \[ \Gamma([L],[M]):= \{\gamma\in\Gamma([L]) : \gamma(1) = [M] \}.\]
    Then $\Gamma([L],[M])$ is isomorphic  to $\{ MP- L : P\in {\bf O}^\ast(L,M)\}/_{{\bf O}^\ast(L,L)}$ and, in particular, the following are equivalent:
  \begin{enumerate}[(i)]
        \item $|\Gamma([L],[M])|=1$;
        \item $\{ MP- L : P\in {\bf O}^\ast(L,M)\}/_{{\bf O}^\ast(L,L)}$ is a singleton.
    \end{enumerate}
\end{corollary}
\begin{proof}
    Let $V\in \{ MP- L : P\in {\bf O}^\ast(L,M)\}$ and set  $\gamma(u):=[L+uV]$ for $u\in[0,1]$.
    Then $\gamma \in \Gamma([L])$ by \Cref{lem:geodesics.one.side} and it is obvious that $\gamma(1)=[M]$, thus $\gamma\in\Gamma([L],[M])$.

    As for the other direction, let $\gamma\in \Gamma([L],[M])$, thus $\gamma(u) = [L+uV]$ for $u\in[0,1]$ and some $V\in {\bf V}(L)$ by \Cref{thm:identification-of-geodesics}, that is, $V=NP-L$ for some $N\in{\bf L}$ and $P\in{\bf O}^\ast(L,N)$.
    Since $\gamma(1)=[M]$ we have $[NP]=[M]$ and therefore $NP=MO$ for some $O\in {\bf O}$. 
    Moreover, using that 
    \[ d_{\rm ABW}([L],[M]) = \|V\|_{\rm F} = \|L-MO\|_{\rm F},\]
    we conclude that $O\in{\bf O}^\ast(L,M)$; thus $V\in \{ M \tilde P- L : \tilde P\in {\bf O}^\ast(L,M)\}$.

    Invariance under ${\bf O}^\ast(L,L)$ follows from \Cref{lem:O^*-facts}.
\end{proof}

\subsection{Alexandrov geometry}
\label{sec:alex.background}

\begin{proof}[Proof of  \Cref{thm:ABW.alex.intro}]
    It follows from \Cref{lem:geodesics.one.side} that $({\bf L}/_{{\bf O}},d_{\rm ABW})$     is a geodesic space and it is obvious that it is complete.

    To show that it is an Alexandrov space with non-negative curvature, it suffices to show the semi-concavity inequality of $d_{\rm ABW}$ along constant speed geodesics, see \eqref{eq:semi.concavity}.
    To this end, let $\gamma$ be a constant speed geodesic starting in $[L]$ represented by $[L + uV] = \gamma(u)$, $u \in [0,1]$, for some $V \in \mathbf V(L)$, see  \Cref{thm:identification-of-geodesics}.
    We have that, for any $M \in \mathbf L$ and $u \in [0,1]$,
    \begin{align*}
        \sup_{O \in \mathbf O} \langle L + u V,MO \rangle_{\rm F} \le (1-u) \sup_{O \in \mathbf O} \langle L,M O\rangle_{\rm F} + u \sup_{O \in \mathbf O} \langle L + V,M O\rangle_{\rm F}.
    \end{align*}
   Next, since $d_{\rm ABW}([L],[L+V])=\|V\|_{\rm F}$ by \Cref{lem:distance.for.V(L)}, a straightforward computation of completing squares shows that
    \begin{multline*}
        d_{\rm ABW}^2([L+uV],[M]) \\ \ge (1-u) d_{\rm ABW}^2([L],[M]) + u d_{\rm ABW}^2([L+V],[M]) - u(1-u) d_{\rm ABW}^2([L],[L+V]),
    \end{multline*}
    which is the desired inequality. 
\end{proof}

We briefly recall the general scheme to define Riemannian-like notions in 
Alexandrov spaces with curvature bounded from below, which will be used in what follows.
For a more in-depth treatment, see, e.g., \cite{AlKaPe24,BuGrPe92}
and also \cite{AmGiSa08} for these concepts applied in the setting of classical optimal transport.

Recall that   $\Gamma([L])$ is the space of all constant speed geodesics emerging from $[L]\in {\bf L}/_{{ \bf O}}$.
For $L\in\bf L$,  $\gamma,\eta\in \Gamma([L])\setminus\{0\}$ (where $0$ refers to the geodesic with zero speed, i.e., $\gamma=0$ if $\gamma(u) = [L]$ for all $u\in[0,1]$), and $u,w\in(0,1]$, we set
  \[   \alpha([L], \gamma(u), \eta(w))
    \coloneqq \frac{d^2_{\rm ABW}([L],\gamma(u)) + d^2_{\rm ABW}([L],\eta(w)) - d^2_{{\rm ABW}}(\gamma(u),\eta(w))}{2d_{{\rm ABW}}([L],\gamma(u))d_{{\rm ABW}}([L],\eta(w))}.\]
Define the angle via 
\[
    \cos(\angle (\gamma, \eta))
     \coloneqq 
   \lim_{u \downarrow 0} \alpha([L], \gamma(u), \eta(u)).
    \]
    Note that $\angle = \angle_{[L]}$ depends on $[L]$, though we follow the standard convention and suppress that dependence. 
	
	Since $(\Gamma([L]), d_{\rm ABW})$ has non-negative curvature, it follows that 
	\[(u,w)\mapsto \alpha([L], \gamma(u), \eta(w))\]
	is nondecreasing, see, e.g., Lemma 12.3.4 in \cite{AmGiSa08}.
	In particular, the limit in the definition of $\cos(\angle (\gamma, \eta))$ exists and
	\[
    \cos(\angle (\gamma, \eta))
    =  \lim_{w \downarrow 0} \lim_{u \downarrow 0} \alpha([L], \gamma(u), \eta(w)).
    \]
    This fact will be used later.

    Next, one can show that $ \angle (\,\cdot,\,\cdot\,)$ induces a pseudometric on $\Gamma([L])\setminus\{0\}$, see e.g.\ \cite[Section 6.5]{AlKaPe24}.
    We denote by $\overline{\Gamma}([L])$ the completion of the quotient space $(\Gamma([L])\setminus\{0\})/_{ \angle}$ with respect to $ \angle (\,\cdot,\,\cdot\,)$.
    In the present framework, it will turn out that $(\Gamma([L])\setminus\{0\})/_{ \angle}$ is already complete, which simplifies the matter.
    Finally, the tangent cone  at $[L] \in \mathbf L /_{\mathbf O}$ denoted by $\mathbf{T}([L])$ is then obtained by 
    taking the metric cone over the completed space $\overline{\Gamma}([L])$, namely 
    \[
\mathbf{T}([L])
\coloneq \operatorname{cone} (\overline{\Gamma}([L])
\coloneq \big(\overline{\Gamma}([L])\times[0,\infty)\big)/_{  {\rm tip}}
\]
 where $(\gamma,r) \sim_{\rm tip} (\eta,s)$ if and only if $r=s=0$, that is, all points $\gamma,\eta\in \overline{\Gamma}([L])$ with attached radius $0$ are identified as the same point.
The cone is equipped with the metric $d_{\mathbf{T}([L])}$, where
\[
d_{\mathbf{T}([L])}^2 \big((\gamma,r),(\eta, s)\big)
:= r^2+s^2-2rs\cos(\angle(\gamma,\eta)).
\]

If $(\Gamma([L])\setminus\{0\})/_{ \angle}$ is already complete, the exponential map $\exp_{[L]}$ at $[L] \in {\bf L}/_{\bf O}$ sends a vector $([\gamma]_{\angle},r)\in \mathbf{T}([L])$ to the endpoint of a constant speed geodesic $\eta$ 
satisfying $d_{\rm ABW}([L],\eta(1))=r$ and $\angle(\gamma,\eta)=0$  -- if such a curve $\eta$ exists.
Note that $\eta$ always exists if $r\leq d_{\rm ABW}([L],\gamma(1))$, but may fail to do so otherwise.

If a point $[M]\in {\bf L}/_{\bf O}$ is reached from  $[L]$ by a unique constant-speed geodesic  $\gamma$, then the log map of $[M]$ at $[L]$  is well defined and given by
\begin{align*}
    \log_{[L]}([M]) = \big([\gamma]_{\angle},\, d_{\mathrm{ABW}}([L],[M])\big) \in {\bf T}([L]).
\end{align*}
In that case,  $\exp_{[L]}(\log_{[L]} ([M] )) = [M] $.

\begin{remark}
It seems plausible that parts of the analysis in this section could also be approached from a Lie-group–theoretic perspective, viewing the action of $\mathbf{O}$ on $\mathbf{L}$ and the associated quotient structure ${\bf L}/_{\bf O}$.
While such an approach might in principle lead to comparable results, obtaining explicit geometric expressions would likely still involve an in-depth treatment. 
In this work, we instead follow an Alexandrov-space viewpoint, which appears to be more natural in the setting of (adapted) optimal transport; in particular, we expect it to be more flexible and better suited for the generalizations beyond the Gaussian setting.
\end{remark}

\subsection{The angle}
The geometry of the tangent cone $\mathbf T([L])$ is governed by the associated distance $d_{\mathbf T([L])}$, 
which in turn depends on computing angles between geodesics.

\begin{theorem}
\label{thm:angle}
	   Fix $L\in \bf L$ and let $\gamma$, $\eta$ be two geodesics in ${\bf L}/_{ {\bf O}}$ emerging from $[L]$ which are represented by $V, W\in {\bf V}(L)\setminus\{0\}$, that is, $\gamma(u) = [ L + uV]$ and  $\eta(u) = [ L + uW]$ for all $u\in[0,1]$.
	Then 
	\[  \cos(\angle(\gamma,\eta)) 
	= \sup_{O \in \mathbf{O}^\ast(L,L)} \frac{ \langle V, W O \rangle_{\rm F} }{ \|V\|_{\rm F}   \| W\|_{\rm F} } .
\] 
\end{theorem}

The proof of \Cref{thm:angle} requires the following observation.

\begin{lemma}
\label{lem:acc.point}
	Let $L\in {\bf L}$, $V,W\in {\bf V}(L)$, and  $w\in [0,1)$.
    Let $(u_n)_{n\geq 1}$ be a sequence converging to 0 and, for each $n\geq1$, let $O^n\in  {\bf O}^\ast(L+u_nV,L+wW)$.
    Then, every $\|\cdot\|_{\rm F}$-accumulation point of $(O^n)_{n\geq 1}$ lies in ${\bf O}^\ast(L, L)$.
\end{lemma}
\begin{proof}
    Let $O$ be such an accumulation point; we assume without loss of generality that $O ^n\to O$.
    Set $M:=L+wW$.
    Since 
    \begin{align*}
     \|L- MO \|_{{\rm F}}
    &=\lim_{n\to\infty} \|L+u_nV- MO^{n} \|_{{\rm F}} \\
    &=\lim_{n\to\infty}  d_{\rm ABW}(L+u_n V, M)
    =d_{\rm ABW}(L, M), 
    \end{align*}
    it follows that $O\in {\bf O}^\ast(L,M)$ and, by \Cref{prop:O.equal.O}, ${\bf O}^\ast(L,M)={\bf O}^\ast(L,L)$.
\end{proof}

\begin{proof}[Proof of \Cref{thm:angle}]
	For $v,w\in (0,1]$, set 
	\[ \mathcal{E}(v,w) := \frac{d_{\rm ABW}^2([L],\gamma(v)) + d_{\rm ABW}^2([L],\eta(w))) - d_{\rm ABW}^2(\gamma(v), \eta(w))}{2d_{\rm ABW}([L],\gamma(v))d_{\rm ABW}([L],\eta(w))},
  \]
  and recall that, by \cite[Lemma 12.3.4]{AmGiSa08}, for $v',w' \in (0,1]$ with $v' \le v$ and $w'\le w$ we have $\mathcal E(v',w') \le \mathcal E(v,w)$.
  Thus, $ \cos(\angle(\gamma,\eta))  = \lim_{v,w\to 0}  \mathcal{E}(v,w)$ and the order of the limits does not matter.
  Observe that, by  \Cref{lem:distance.for.V(L)}, $d_{\rm ABW}([L],\gamma(v)) = v \|V\|_{\rm F}$ and $d_{\rm ABW}([L],\eta(w))=  w\| W\|_{\rm F}$ for all $v,w$, hence
  \[ \mathcal{E}(v,w) = \frac{ v^2 \|V\|_{\rm F}^2+ w^2 \|W\|_{\rm F}^2 - d_{\rm ABW}^2(\gamma(v),\eta(w)))}{2vw\|V\|_{\rm F} \|W\|_{\rm F}}.
  \]

	\noindent
	\emph{Step 1:} We first claim that 
	\begin{align}
	\label{eq:angle.geq}
	\cos(\angle(\gamma,\eta)) 
	\geq  \sup_{O \in \mathbf{O}^\ast(L,L)} \frac{ \langle V,  W O \rangle_{\rm F} }{ \|V\|_{\rm F}   \|W \|_{\rm F} }.
	\end{align}
	To that end, note that, by the definition of $d_{\rm ABW}$,
	\begin{align*} 
	d_{\rm ABW}(\gamma(v),\eta(w))
	\leq \inf_{O\in {\bf O}^\ast(L,L) }  \|L + vV - (L+ w W)O\|_{\rm F}.
	\end{align*}
	Since $L=LO$ for every $O\in {\bf O}^\ast(L,L)$, 
		\begin{align*} 
	d_{\rm ABW}^2(\gamma(v),\eta(w))
	&\leq \inf_{O\in {\bf O}^\ast(L,L) }  \| vV - w WO\|_{\rm F}^2 \\
	&=   v^2\| V\|_{\rm F}^2 + w^2 \| W \|_{\rm F}^2  - 2 vw  \sup_{O\in {\bf O}^\ast(L,L) } \langle V,WO\rangle_{\rm F},
	\end{align*}
	from which \eqref{eq:angle.geq} follows.
    \medskip
	
	\noindent
	\emph{Step 2:} We proceed to show that
	\begin{align*}
	\label{eq:angle.leq}
	\cos(\angle(\gamma,\eta)) 
	\leq  \sup_{O \in \mathbf{O}^\ast(L,L)} \frac{ \langle V, WO \rangle_{\rm F} }{ \|V\|_{\rm F}   \| W\|_{\rm F} }.
	\end{align*}
	 A twofold application of  the polarization identity  shows that, for every $O\in {\bf O}$, 
	   \begin{align*}
	   \|(L + vV) - &(L + wW)O \|_{\rm F}^2 \\
	  &=    \|L + vV\|_{\rm F}^2 + \|L + wW \|_{\rm F}^2 - 2 \langle L + vV, (L + wW)O \rangle_{\rm F} \\
	 &  =  v^2  \|V\|_{\rm F}^2 + w^2\| W \|_{\rm F}^2  +  2\langle L, L+ vV +w W \rangle  - 2 \langle L + vV, (L + wW)O \rangle_{\rm F}.
	\end{align*}
	It follows that
	\[  \mathcal{E}(v,w)
	=    \frac{-\langle L, L + vV + w W \rangle_{\rm F} + \inf_{O \in \mathbf{O}} \langle L + vV, (L + wW)O \rangle_{\rm F}}{vw \|V\|_{\rm F} \|W\|_{\rm F}}. \]
	 Denote by $O^{v,w}\in \mathbf{O}$ an optimizer for the last term, that is,
	 \[O^{v,w}\in {\bf O}^\ast( L + vV, L + wW).\]
	 By Proposition \ref{prop:O.equal.O} we have that ${\rm Id} \in  {\bf O}^\ast( L , L + wW)$, hence
	 \[
     \langle L, L + w W \rangle_{\rm F} = \sup_{O \in \mathbf{O}} \langle L, (L + wW) O \rangle_{\rm F} 
	 \geq \langle L, (L + wW) O^{v,w} \rangle_{\rm F}, \]
	 from which it follows that
	 \[
     \mathcal{E}(v,w)
	\leq   \frac{  - \langle L,V \rangle_{\rm F}+    \langle V, (L + wW)O^{v,w} \rangle_{\rm F} }{w \|V\|_{\rm F}\|W\|_{\rm F}}. 
    \]
	 Fix $w\in(0,1]$.
	 Clearly $(O^{v,w})_{v\in(0,1)}$ is relatively compact, and by Lemma \ref{lem:acc.point}, any accumulation point  $O$ of $(O^{v,w})_{v\in(0,1)}$ as $v \to 0$ must lie in $\mathbf{O}^\ast(L,L)$. In particular,  $ L O = L$.
	Thus 
	 \[ \limsup_{v\to 0} \mathcal{E}(v,w)
	 \leq \lim_{v \to 0} \frac{ -\langle L, V\rangle +  \langle V, (L + w W)O^{v,w} \rangle_{\rm F} }{w \|V\|_{\rm F} \| W \|_{\rm F}}
     =  \frac{  \langle V, W O\rangle_{\rm F} }{\|V\|_{\rm F} \| W \|_{\rm F}},\]
	 from which the claim follows.
\end{proof}

\subsection{The tangent cone}
Let $\gamma,\eta\in\Gamma([L])$ be two geodesics starting from $[L]$ which are induced by $V$ and $W$ as in \Cref{thm:identification-of-geodesics}.
Then, by \Cref{thm:angle},  $\angle(\gamma,\eta)=0$ (i.e., $\gamma\sim_{\angle}\eta$) if and only if there is $O\in {\bf O}^\ast(L,L)$ and $\lambda>0$ for which  $V=\lambda WO$.

Recall that 
\[ \mathcal{V}(L) =  \left\{ V \in {\bf L}: (V^\top L)_{t,t} \text{ is symmetric for all } t = 1, \dots, T\right\}. \]
Clearly $\mathcal{V}(L)$ is a linear space.
To show that (after taking an appropriate quotient) $\mathcal{V}(L)$ is isometric to the tangent cone, we require some preliminary steps.

For a symmetric positive semi-definite matrix $A$, we set $\lambda_{\min}^+(A)$ 
to be $+\infty$ if $A=0$ and its smallest strictly positive eigenvalue otherwise.  

\begin{lemma} \label{lem:V-psd}
Let $L \in \mathbf L$ and $V \in \mathcal V(L)$ such that for all $t = 1, \dots, T$ we have 
\[
\|V_{\cbullet, t}\|_{\rm op}^2
<  \lambda_{\min}^+ \left((L^\top L)_{t,t}\right),\]
where $\|\cdot\|_{\rm op}$ is the operator norm.
Then $((L + V)^\top L)_{t,t}$ is positive semi-definite and
\[
    {\bf O}^\ast(L,L+V) = {\bf O}^\ast(L, L).
\]
In particular, if $L\in {\bf L}^{\rm reg}$ then 
$\mathbf O^*(L, L + V) = \{ {\rm Id} \}$.
\end{lemma}
\begin{proof}
    Let $t = 1, \dots, T$ be arbitrary and note that 
    \begin{equation}
    \label{eq:MLexpanded}
    \left(\left(L+V\right)^\top L \right)_{t,t} = (L^\top L)_{t,t} + (V^\top L)_{t,t}
    \end{equation}
    is a symmetric matrix since $( V^\top L)_{t,t}$ is symmetric. 
    We want to apply \Cref{lem:O.equal.O.2}, for which we need to show that $((L + V)^\top L)_{t,t}$ 
    is positive semi-definite and 
    ${\rm rk}((L^\top L)_{t,t}) = {\rm rk}(((L+V)^\top L)_{t,t})$ for all $t = 1, \dots, T$.
    To prove this claim, observe that for any $x \in \R^d$
    \begin{align}
    \nonumber
    x^\top \left((L+V)^\top L \right)_{t,t} x
    &= x^\top ( L^\top L)_{t,t} x +  
    x^\top ( V^\top L)_{t,t} x
    \\  \nonumber
    &= \|L_{\cbullet, t}x \|_2^2 + \langle V_{\cbullet, t} x, L_{\cbullet, t} x \rangle
    \\ \label{eq:cal.V.psd}
    &\geq \left(\|L_{\cbullet, t}x \|_2 - \| V_{\cbullet, t} x\|_2 \right) \|L_{\cbullet, t} x \|_2,
    \end{align}
    where we used the Cauchy-Schwartz inequality in the last line.
    Now, for $x \in \ker(L_{\cbullet, t})$, the RHS of \eqref{eq:cal.V.psd} is clearly non-negative. 
    In particular, if  $(L^\top L)_{t,t}=0$, the claim immediately follows.
        Otherwise, let $x \in \ker(L_{\cbullet, t})^\perp$ and note that
    \begin{equation} \label{eq:psd-inequ}
        \left(\|L_{\cbullet, t}x \|_2 - \| V_{\cbullet, t} x\|_2 \right)
        \geq \left( \sqrt{\lambda_{\rm min}^+((L^\top L)_{t,t})} 
            - \|V_{\cbullet, t}\|_{\rm op}\right) \|x\|_2,
    \end{equation}
    which is non-negative by assumption.
    Thus $((L+V)^\top L)_{t,t}$ is positive semi-definite.

    To prove that ${\rm rk}((L^\top L)_{t,t}) = {\rm rk}(((L+V)^\top L)_{t,t})$, by \Cref{rmk:rank-im-ker}, it suffices to prove that   ${\rm ker}((L^\top L)_{t,t}) = {\rm ker}(((L+V)^\top L)_{t,t})$.
    To that end, note that by \eqref{eq:MLexpanded} 
    \[\ker((L^\top L)_{t,t}) \subseteq \ker(((L+V)^\top L)_{t,t}).\]
    The other inclusion follows from  \eqref{eq:cal.V.psd}  and \eqref{eq:psd-inequ}.
   
   We conclude that the assumptions of \Cref{lem:O.equal.O.2} are satisfied for $M:= L+V$, from which the claim follows. 
\end{proof}

\begin{proposition}
\label{prop:cone.VL}
    For every $L\in{\bf L}$, we have that
    \[ \mathcal{V}(L) = \left\{ \lambda V : \lambda \geq 0 , \, V\in {\bf V}(L)\right \}.  \]
\end{proposition}
\begin{proof}
    Let $V\in {\bf V}(L)$, i.e., $V=MP-L$ for some $M\in{\bf L}$ and $P\in {\bf O}^\ast(L,M)$, and let $\lambda\geq0$.
    By \Cref{lem:ABW-optimizers}, $P^\top_t (M^\top L)_{t,t}$ is symmetric, 
    hence $V\in \mathcal{V}(L)$, and as $\mathcal{V}(L)$ is a subspace, $\lambda V\in\mathcal{V}(L)$.

    For the other inclusion, let $V \in \mathcal V(L)$. It is clear that there exists a constant $\lambda > 0$ such that $\tilde V := \lambda V$ satisfies the assumptions of \Cref{lem:V-psd}. 
    Thus ${\rm Id} \in \mathbf O^\ast(L, L+\tilde V)$, which implies 
    \[\tilde V = (\tilde V + L){\rm Id} - L \in \mathbf V(L).\]
    As $V = \frac{1}{\lambda}\tilde V$, the claim follows.
\end{proof}

\begin{proposition}
\label{prop:complete}
	For every $L\in{\bf L}$, $((\Gamma([L])\setminus\{0\})/_{\angle}\,, \, \angle )$ is complete.
    Moreover, setting
    \[ \mathcal{S} := \left\{ \frac{V}{\|V\|_{\rm F}} : V\in {\bf V}(L)\setminus\{0\} \right\},\]
    we have that $\overline{\Gamma}([L])$ is isomorphic to $\mathcal{S}/_{{\bf O}^\ast(L,L)}$.    
\end{proposition}

\begin{proof}
    We first show the isomorphism
    \[(\Gamma([L]) \setminus\{0\})/_{\angle}  \cong \mathcal{S}/_{{\bf O}^\ast(L,L)}.\]
   Indeed, let $\gamma,\eta\in\Gamma([L])$ be represented by $V,W\in {\bf V}(L)\setminus\{0\}$, respectively, and set $\overline{V}:=V/\|V\|_{\rm F}$ and $\overline{W}:= W  / \| W\|_{\rm F}$.
    Theorem \ref{thm:angle} implies that $\angle(\gamma,\eta) =0$ if and only if there are $O\in{\bf O}^\ast(L,L) $     and $\lambda \geq 0$ such that $V=\lambda W O $. 
    The latter is equivalent to $\overline{V}= \overline{W}O$.    
    Conversely, if $V,W \in\mathcal{S}$ satisfy $V=  WO$ for some $O \in{\bf O}^\ast(L,L)$, then the associated geodesics $\gamma,\eta \in \Gamma([L])$ satisfy $\angle(\gamma,\eta)=0$.

    Next, we claim that  $(\mathcal{S},\|\cdot\|_{\rm F})$ is compact, in which case completeness of $((\Gamma([L])\setminus\{0\})/_{\angle},\angle)$ clearly follows since $\angle$ is continuous with respect to  $\|\cdot\|_{\rm F}$.
    To prove compactness, let $(W^n)_{n \geq 1}$ be a sequence in $\mathcal{S}$, and write 
    \[W^n=\frac{V^n}{\|V^n\|_{\rm F}}, \quad \text{for }V^n\in { \bf V}(L).\]
    By \Cref{prop:cone.VL}, we have that ${\bf V}(L)\subseteq \mathcal{V}(L)$ and the latter is a cone, hence $W^n\in\mathcal{V}(L)$ for all $n\geq 1$.
    Since $\mathcal{V}(L)$ is closed w.r.t.\ $\|\cdot \|_{\rm F}$, there exist $W\in\mathcal{V}(L)$ and a subsequence $(n_k)_{k\geq 1}$ satisfying $\|W^{n_k}-W\|_{\rm F} \to 0$ as $k\to\infty$.
    By another application of Proposition \ref{prop:cone.VL}, there are $V\in {\bf V}(L)$ and $\lambda\geq 0$ satisfying $W=\lambda V$.
    Since $\|W\|_{\rm F}=1$, clearly $\lambda = 1/\|V\|_{\rm F}$ and thus $W\in\mathcal{S}$.
\end{proof}

We are now ready to prove Theorem~\ref{thm:tangent.intro}, 
the isometry between the tangent cone of $\mathbf L/_{\mathbf O}$ and the appropriate quotient of $\mathcal V(L)$:

\begin{proof}[Proof of Theorem~\ref{thm:tangent.intro}]
    Using the notation of Proposition \ref{prop:complete}, $\overline{\Gamma}([L]) \cong \mathcal{S}/_{{\bf O}^\ast(L,L)}$.
    Moreover, by Proposition \ref{prop:cone.VL} we have that 
    $ \{\lambda W : W\in\mathcal{S}, \lambda \geq 0\} = \mathcal{V}(L)$,
    from which it readily follows that 
    \[  {\bf T}([L]):={\rm cone} (\overline{\Gamma}([L])) \cong \mathcal{V}(L) /_{{\bf O}^\ast(L,L)}.\]
    Thus, by the definition of  $d_{ {\bf O}^\ast(L,L)}$, followed by Theorem \ref{thm:angle} and Pythagorean theorem,
    \begin{align*}
        d_{ {\bf T}([L])}^2(V,W) 
        &= \|V\|_{\rm F}^2 + \| W\|_{\rm F}^2 - 2 \sup_{O\in{\bf O}^\ast(L,L)} \langle V, WO \rangle_{\rm F} 
    = d_{{\bf O}^\ast(L,L)} ^2(V,W),
    \end{align*}
    as claimed.
\end{proof}

Finally, we characterize when  a unique geodesic between $[L]$ and $[M]$  exists.
Recall that, by the definition of the exponential map, for every $V\in {\bf V}(L)$,  
\begin{align}
    \label{eq:exponential.structure}
    \exp_{[L]}(\iota_L^{-1}([V]_{{\bf O^\ast}(L,L)})) = [L+V],
\end{align}
and by \Cref{cor:unique.geo}, there is a unique geodesic between $[L]$ and $[M]$ if and only if $\{ MP- L: P\in  {\bf O}^\ast(L,M)\}/_{{\bf O^\ast}(L,L)}$ is a singleton.
In that case, 
    \begin{align}
    \label{eq:log.structure}
    \log_{[L]}([M]) = \iota_L^{-1}( [MP-L]_{{\bf O^\ast}(L,L)}) 
    \end{align}
    for some $P\in {\bf O}^\ast(L,M)$.

\begin{lemma} \label{lem:log-map-uniqueness}
    For every $L,M \in \mathbf L$, the following are equivalent:
    \begin{enumerate}[(i)]
        \item $\{MP - L : P \in \mathbf O^\ast(L,M) \}/_{\mathbf O^\ast(L,L)}$ is a singleton;
        \item $\operatorname{rk}((M^\top M)_{t,t}) = \operatorname{rk}((M^\top L)_{t,t})$ for all $t=1,\dots,T$.
    \end{enumerate}
\end{lemma}

\begin{proof}
If follows from \Cref{lem:O^*-facts} that 
for all $O \in \mathbf O^\ast(L,L)$,
\[
\{MP - L : P \in \mathbf O^\ast(L,M) \}
=
\{(MP - L)O : P \in \mathbf O^\ast(L,M) \}.
\]
Therefore (i) is equivalent to
$\{MP: P \in \mathbf O^\ast(L,M) \}$ being a singleton.
We claim that the latter is equivalent to (ii).
Indeed, $\{MP: P \in \mathbf O^\ast(L,M) \}$ is a singleton
if and only if for any reference optimizer $P \in \mathbf O^\ast(L,M)$
we have $MP = M \tilde P$ for all $\tilde P \in \mathbf O^\ast(L,M)$.
Multiplying the equation by $P^\top$, 
this is equivalent to 
\begin{equation} \label{eq:MeqMP}
    M = M \overline P \text{ for all }\overline P \in \mathbf O^\ast(L P^\top,M),
\end{equation}
where we used the fact that $\mathbf O^\ast(L,M) \cdot P^\top = \mathbf O^\ast(L P^\top,M)$.

Next, since $(M^\top L)_{t,t}P_t^\top$ is symmetric for all $t = 1, \dots, T$, \Cref{lem:O^*-facts} (iv) implies that
$\mathbf O^\ast(L P^\top,M) = \mathbf O^\ast(M, L P^\top)$.
Therefore, \eqref{eq:MeqMP} is equivalent to $\mathbf O^\ast(M, L P^\top) \subseteq \mathbf O^\ast(M,M)$.
Moreover by \Cref{lem:O^*-facts} (ii) we always have that $\mathbf O^\ast(M,M)\subseteq \mathbf O^\ast(M, L P^\top)$, hence \eqref{eq:MeqMP} is equivalent to $\mathbf O^\ast(M, L P^\top) = \mathbf O^\ast(M,M)$.

Finally, by \Cref{lem:O.equal.O.2}, the latter  is equivalent to $((L P^\top)^\top M)_{t,t}$ being positive semi-definite 
and 
\[ \operatorname{rk}\left((M^\top M)_{t,t}\right) 
= \operatorname{rk}\left((L P^\top)^\top M)_{t,t}\right)\]
for every $t = 1, \dots, T$.
\end{proof}

\begin{proof}[Proof of Theorem \ref{thm:exponential.general.intro}]
    This follows from \Cref{cor:unique.geo}, \Cref{lem:log-map-uniqueness}, \eqref{eq:exponential.structure}, and \eqref{eq:log.structure}.
\end{proof}

\subsection{The tangent space for $L\in {\bf L}^{\rm reg}$}
We proceed to prove the statements in the introduction pertaining ${\bf L}^{\rm reg}$, 
the subset corresponding to Gaussian processes that are regular in the adapted sense.

\begin{proof}[Proof of \Cref{thm:L.reg.convex.subset.L.intro}]
    Let $L,M \in \mathbf L^{\rm reg}$. 
    By \Cref{thm:identification-of-geodesics} and \eqref{eq:V-def}, there is $P \in {\bf O}^\ast(L,M)$ such that 
    \[ \gamma(u) = [L(u)] , \quad \text{where } L(u)=(1-u)L + uMP,  \quad u\in[0,1].\]
    Thus
    \begin{align*}
    (L(u)^\top L)(u)_{t,t}
        &=(1-u)^2 (L^\top L)_{t,t}
    + 2 u (1-u)(P^\top M^\top L)_{t,t}
    + u^2 ( P^\top M^\top M P)_{t,t},
\end{align*} 
    where we used that $(L^\top MP)_{t,t} = (P^\top M^\top L)_{t,t}$. 
    Note that both 
    $(L^\top L)_{t,t}$ and $(M^\top M)_{t,t}$
    are positive definite for all $t = 1, \dots, T$,
    hence  $(L(u)^\top L(u))_{t,t}$ is positive definite for every $t = 1, \dots, T$,  and $\gamma(u)\in \mathbf L^{\rm reg}/_{\bf O}$ for all $u \in [0,1]$, as claimed.
    
    The structure of the tangent cone is a consequence of \Cref{thm:tangent.intro} and the fact that, for $L\in {\bf L}^{\rm reg}$, ${\bf O}^\ast(L,L)=\{\rm Id\}$ (see \Cref{cor:O=Id}); hence 
    \[ \left( \mathcal{V}(L)/_{{\bf O}^\ast(L,L)}, d_{{\bf O}^\ast(L,L)} \right)
    = (\mathcal{V}(L),\|\cdot\|_{\rm F} ). \qedhere\]
\end{proof}

\begin{proof}[Proof of \Cref{prop:L.reg.exponential.intro}]
    This follows from  \Cref{prop:cone.VL}, \eqref{eq:exponential.structure}, and since  ${\bf O}^\ast(L,L)=\{\rm Id\}$ for $L\in {\bf L}^{\rm reg}$.
\end{proof}

\begin{proof}[Proof of \Cref{thm:L.reg.uniq.geodesic}]
    This follows from \Cref{thm:exponential.general.intro} and since ${\bf O}^\ast(L,L)=\{\rm Id\}$ for $L\in {\bf L}^{\rm reg}$.
\end{proof}

\begin{proof}[Proof of \Cref{cor:L.reg.uniq.geodesic}]
    For $P\in {\bf O}^\ast(L,M)$, consider $V:=MP-L \in \mathbf V(L)$.
    Then, for all $t = 1, \dots, T$,
    \begin{align*}
     \|V_{\cbullet, t} \|_{\rm op}^2
    \leq \|V \|_{\rm F}^2 
    &= d_{\rm ABW}^2([L],[M]) \\
    &<  \min_{s=1,\dots,T}  \lambda_{\min}( (L^\top L)_{s,s})
    \leq
    \lambda_{\min}( (L^\top L)_{t,t}),
    \end{align*}
    where we used \Cref{lem:distance.for.V(L)} for the equality. 
   Therefore, by \Cref{lem:V-psd} we have ${\bf O}^\ast(L,L+V)=\{ \rm Id\}$.
    Since $L+V= MP$,  it  follows that ${\bf O}^\ast(L,M)=\{ P\}$. 
    In particular,  $\{MP-L : P\in {\bf O}^\ast(L,M)\}$ is a singleton and the claim follows from \Cref{cor:unique.geo}.
\end{proof}

\begin{remark}
One could further pursue the geometric analysis of $\mathbf L/_\mathbf O$ within the framework of Alexandrov geometry.
To explain this perspective, recall that for any point $[L]\in \mathbf L^{\rm reg}/_\mathbf O$,
the tangent cone $T([L])$ is isometric to the linear space $\mathcal V(L)$,
which is an Euclidean space of dimension $d_T := \frac{Td(Td+1)}{2}$.
In particular, $\mathbf L^{\rm reg}/_{\mathbf O}$ consists precisely of the so-called regular points of $\mathbf L/_{\mathbf O}$ (i.e.\ points where tangent cones are isometric to Euclidean spaces of matching dimension) and, trivially, forms a dense open subset. 
Under these conditions, if follows from \cite[Theorem~15.13]{AlKaPe24} that $d_T$ coincides with the (linear) dimension of the Alexandrov space $\mathbf L/_{\mathbf O}$.

Finite-dimensional Alexandrov spaces with curvature bounded from below are well studied; see e.g.\ \cite[Chapter~10]{BuBuIv01}. 
In particular, by Theorem~10.10.1 therein, $\mathbf L^{\rm reg}/_{\mathbf O}$ is a $d_T$-dimensional topological manifold. 
While \cite{OtSh94, Pe94} and subsequent work indicate that the subset of regular points of finite-dimensional Alexandrov spaces carries additional geometric structure, we leave this for future research.
Our focus is instead on settings more likely admitting infinite-dimensional generalizations, 
specifically to subspaces of filtered processes $({\rm FP}_2,\AW_2)$ other than ${\rm GP}$, 
where finite-dimensional Alexandrov techniques are not immediately available.
\end{remark}

\subsection*{Acknowledgment}

D.\ Bartl and A.\ Grass are grateful for support by the the Austrian National Bank [Jubil\"aumsfond, project 18983]. 
D.\ Bartl is furthermore grateful for support by the Austrian Science Fund [doi: 10.55776/P34743 and 10.55776/ESP31] 
and a Presidential Young Professorship grant [`Robust statistical learning for complex data'].

\printbibliography

\end{document}